\documentclass[12pt,reqno]{amsart}
\usepackage{amssymb}
\usepackage{graphicx}
\usepackage{xcolor}
\usepackage{bbm}

\usepackage[all]{xy}
\DeclareFontFamily{U}{mathb}{\hyphenchar\font45}
\DeclareFontShape{U}{mathb}{m}{n}{
      <5> <6> <7> <8> <9> <10> gen * mathb
      <10.95> mathb10 <12> <14.4> <17.28> <20.74> <24.88> mathb12
      }{}
\DeclareSymbolFont{mathb}{U}{mathb}{m}{n}
\DeclareMathSymbol{\righttoleftarrow}{3}{mathb}{"FD}

\theoremstyle{plain}
\newtheorem{prop}{Proposition}[section]
\newtheorem{theo}[prop]{Theorem}
\newtheorem{coro}[prop]{Corollary}
\newtheorem{lemm}[prop]{Lemma}
\theoremstyle{remark}
\newtheorem{rema}[prop]{Remark}

\theoremstyle{definition}
\newtheorem{defi}[prop]{Definition}

\newtheorem{exam}[prop]{Example}
\numberwithin{equation}{section}

\newcommand{\A}{{\mathbb A}}
\newcommand{\C}{{\mathbb C}}
\newcommand{\PP}{{\mathbb P}}
\newcommand{\bP}{{\mathbb P}}
\newcommand{\Q}{{\mathbb Q}}
\newcommand{\F}{{\mathbb F}}
\newcommand{\G}{{\mathbb G}}
\newcommand{\N}{{\mathbb N}}

\newcommand{\Z}{{\mathbb Z}}

\newcommand{\cB}{{\mathcal B}}

\newcommand{\cD}{{\mathcal D}}
\newcommand{\cE}{{\mathcal E}}
\newcommand{\cF}{{\mathcal F}}

\newcommand{\cO}{{\mathcal O}}
\newcommand{\cI}{{\mathcal I}}

\newcommand{\cK}{{\mathcal K}}
\newcommand{\cL}{{\mathcal L}}
\newcommand{\cM}{{\mathcal M}}
\newcommand{\cN}{{\mathcal N}}
\newcommand{\cS}{{\mathcal S}}
\newcommand{\cT}{{\mathcal T}}
\newcommand{\cX}{{\mathcal X}}
\newcommand{\cY}{{\mathcal Y}}

\newcommand{\cU}{{\mathcal U}}

\newcommand{\cW}{{\mathcal W}}

\newcommand{\fA}{{\mathfrak A}}
\newcommand{\sD}{{\mathfrak D}}
\newcommand{\fK}{{\mathfrak K}}

\newcommand{\fs}{{\mathfrak s}}
\newcommand{\BC}{{\mathcal B}{\mathcal C}}

\newcommand{\rB}{{\mathrm B}}
\newcommand{\rH}{{\mathrm H}}

\newcommand{\rN}{{\mathrm N}}

\newcommand{\iJ}{{\mathrm{IJ}}}
\newcommand{\GL}{{\mathrm{GL}}}
\newcommand{\PGL}{{\mathrm{PGL}}}

\newcommand{\ra}{\rightarrow}
\newcommand{\lra}{\longrightarrow}

\newcommand{\bA}{{\mathbb A}}
\newcommand{\bC}{{\mathbb C}}
\newcommand{\bF}{{\mathbb F}}
\newcommand{\bG}{{\mathbb G}}

\newcommand{\bN}{{\mathbb N}}
\newcommand{\bQ}{{\mathbb Q}}
\newcommand{\bR}{{\mathbb R}}
\newcommand{\bZ}{{\mathbb Z}}

\newcommand{\Am}{\mathrm{Am}}

\newcommand{\fS}{{\mathfrak S}}

\newcommand{\eqto}{\stackrel{\lower1.5pt\hbox{$\scriptstyle\sim\,$}}\to}
\newcommand{\eqdashto}{\stackrel{\lower1.5pt\hbox{$\scriptstyle\sim\,$}}\dashrightarrow}
\newcommand{\actsfromleft}{\mathrel{\reflectbox{$\righttoleftarrow$}}}
\newcommand{\actsfromright}{\righttoleftarrow}
\DeclareMathOperator{\Alg}{Alg}
\DeclareMathOperator{\Gal}{Gal}

\DeclareMathOperator{\Pic}{Pic}
\DeclareMathOperator{\Spec}{Spec}

\DeclareMathOperator{\Hom}{Hom}
\DeclareMathOperator{\End}{End}
\DeclareMathOperator{\Br}{Br}
\DeclareMathOperator{\Aut}{Aut}
\DeclareMathOperator{\Ker}{Ker}

\DeclareMathOperator{\Var}{Var}
\DeclareMathOperator{\Burn}{Burn}
\DeclareMathOperator{\Bir}{Bir}
\DeclareMathOperator{\Cr}{Cr}

\DeclareMathOperator{\Ind}{Ind}

\newcommand{\oBurn}{{\overline\Burn}}

\begin{document}
\title[Equivariant birational geometry]{Invariants in equivariant birational geometry}

\author[Andrew Kresch]{Andrew Kresch}
\address{
  Institut f\"ur Mathematik,
  Universit\"at Z\"urich,
  Winterthurerstrasse 190,
  CH-8057 Z\"urich, Switzerland
}
\email{andrew.kresch@math.uzh.ch}

\author[Yuri Tschinkel]{Yuri Tschinkel}
\address{
  Courant Institute,
  251 Mercer Street,
  New York, NY 10012, USA
}
\email{tschinkel@cims.nyu.edu}

\address{Simons Foundation\\
160 Fifth Avenue\\
New York, NY 10010\\
USA}

\date{February 20, 2026}

\begin{abstract}
We discuss invariants in equivariant birational geometry.
\end{abstract}

\maketitle

\section{Introduction}
\label{sec.intro}
In this survey, we consider problems in classical birational geometry and their extensions to the setting of equivariant geometry. The focus is on new invariants, introduced in \cite{BnG}, building on \cite{KT}, \cite{KPT}, \cite{Bbar}, and 
\cite{HKTsmall}. 
In our applications, we work over an algebraically closed ground field $k$ of characteristic zero. 

The term {\em birational geometry} refers to the characterization of isomorphism classes of function fields of algebraic varieties over $k$, i.e., fields $K$ which are finitely generated over $k$. 
Of particular interest is the {\em  (stable) rationality problem}, i.e., 
the characterization of {\em (stably) rational} function fields. 
Recall that an algebraic variety $X$ over $k$ is called {\em rational}, respectively, {\em stably rational}, if its function field $K=k(X)$, respectively, the field $K(x_1,\ldots, x_n)$ for suitable $n$, 
is isomorphic to a purely transcendental extension of $k$. 
An interesting related notion is that of {\em unirationality}, 
i.e., the property that $K\subseteq k(x_1,\ldots, x_n)$.
The (stable) rationality problem is settled in dimension 2, but remains elusive in dimensions $\ge 3$. 

One can ask about birationality of varieties with group actions. Concretely, let $G$ be a finite group, acting on $K=k(X)$, and trivially on $k$. If the action is generically free, then $G$ is identified with a subgroup of 
$$
\Bir\Aut(X),
$$
the group of birational automorphisms of $X$ over $k$. The main question in this context is: 

\medskip

\noindent
{\em When are two such subgroups conjugate in $\Bir\Aut(X)$?}

\medskip

There is always a smooth projective model for which the action is \emph{regular} (see Section \ref{sect:basic}).
Taking $X$ to be such a model, we can view $G$ as a subgroup of 
$$
\Aut(X),
$$
the group of biregular automorphisms of $X$. 
For example, when $X=\bP^2$, the projective plane, we have $\Aut(X)\simeq \PGL_3$, and it is known classically how to distinguish finite subgroups up to conjugation in $\PGL_3$.
But the problem of distinguishing finite subgroups of $\PGL_3$ up to conjugation in the \emph{Cremona group} 
$$
\mathrm{Cr}_2:=\Bir\Aut(\bP^2),
$$ 
the group of birational automorphism of $\bP^2$, is a recent advance \cite{CTZ-p2}.

In this paper, we discuss new 
invariants in $G$-equivariant birational 
geometry, and in particular, the 
{\em equivariant Burnside group}
$$
\Burn_n(G),
$$
introduced in \cite{BnG}. 
Needless to say, there 
is an enormous literature on equivariant birational geometry; we will not be able to give a complete account of all important ideas and techniques. We will point to sources that should allow an interested reader to learn more about this fascinating subject. 

One of the motivations to study $G$-birational geometry lies in strongly suggestive analogies with birational geometry over nonclosed fields, a booming subject. 
These analogies have been emphasized in \cite{manin} and \cite{manin2}:
the role of $G$ is played by Galois symmetries, which act on special loci, geometric invariants, etc. 
There are also appreciable differences between these theories. We will not pursue this theme here, but refer to \cite{HTtors}, \cite{HT-quad}, \cite{KT-dp} for recent developments inspired by this dual point of view. 

\medskip

Here is a roadmap of this paper.
In Sections \ref{sec.key} and \ref{sect:basic} we discuss the basics of equivariant birational geometry and the evolution of ideas and constructions that led us to the definition of $\Burn_n(G)$:
\begin{itemize}
\item {\em analytic tools}: motivic integration and (specialization of) rationality,
\item {\em geometric tools}: equivariant weak factorization, birational rigidity, stacks. 
\end{itemize}
We give a short summary of known obstructions in equivariant birational geometry in  Section~\ref{sect:obstr}. In Section~\ref{sect:amitsur}, we discuss the Amitsur group and introduce a higher version of the classical Amitsur invariant, based on the Leray spectral sequence for  motivic complexes of stacks.
In Section~\ref{sect:model}, we explain how to choose 
suitable birational models for $G$-actions. In Section~\ref{sect:defi}, we recall the definition of equivariant Burnside groups from \cite{BnG} and explain how to compute the class of the action. 
We comment on first structural properties of these new invariants in Section~\ref{sect:first}: 
\begin{itemize}
\item {\em filtrations},
\item {\em incompressible symbols},
\item {\em specialization}.
\end{itemize}
We present several applications in Section~\ref{sect:apply} and discuss limitations of these invariants towards the end of the paper, in Section~\ref{sect:scifi}.

\bigskip

\noindent
\textbf{Acknowledgments:}
The authors thank Joseph Ayoub for extensive discussions and essential input for the treatment of motivic cohomology.
The authors are very grateful to Brendan Hassett for his interest and collaboration on related projects. The second author was partially supported by NSF grant 2301983.

\section{Key ideas}
\label{sec.key}
Throughout, we work over a field $k$ of characteristic zero. 

\subsection*{Motivic integration}
Motivic integration was introduced by Kontsevich in his generalization of Batyrev's theorem 
\cite{batyrev}: birational Calabi-Yau varieties have equal {\em Hodge numbers}. 
It is now a major research direction, with many applications, e.g., to the study of singularities, McKay correspondence, etc.; see \cite{CLNS}.

Let 
$$
\Var_{n,k}
$$
be the set of isomorphism classes of algebraic varieties over $k$ of dimension $n$, and put
$$
\Var_k = \bigsqcup_{n\ge 0} \Var_{n,k},
\qquad
\Var_k^{\le d} = \bigsqcup_{0\le n\le d} \Var_{n,k}.
$$
We will write $[X]$ for the class of an algebraic variety (i.e., a reduced, separated, finite-type $k$-scheme) in $\Var_k$. 
The {\em Grothendieck ring}
$$
\mathrm K_0(\Var_k)
$$
is defined as the quotient of the $\bZ$-module 
$$
\bZ[\Var_k]
$$
by the relations:
\begin{itemize}
\item {\em Excision}:  $[X]=[Z]+[X\setminus Z]$, for closed $Z\subset X$,
\item {\em Product}: $[X\times Y] = [X]\times [Y]$, for all $X$, $Y$.
\end{itemize}
The ring $\mathrm K_0(\Var_k)$ carries a natural filtration, by the images of $\bZ[\Var_k^{\le d}]$. 

A {\em motivic measure} with values in a commutative ring $R$ is a ring homomorphism 
$$
\mu\colon \mathrm K_0(\Var_k) \to R.
$$
For example, for $k=\C$ we take $R=\bZ[u,v]$ and 
define, for smooth projective $X$,  
$$
\mu([X]) := \sum_{p,q} (-1)^{p+q}\,  h^{p,q}(X)u^pv^q,
$$
the {\em Hodge-Deligne polynomial}. 
Kontsevich's refinement of Batyrev's theorem shows that 
birational Calabi-Yau varieties have the same class in $\mathrm K_0(\Var_k)$, thus the same Hodge-Deligne polynomials, and thus the same Hodge numbers.

\subsection*{Stable rationality and $\mathrm K_0(\Var_k)$}
The beautiful paper \cite{larsenlunts} revealed an unexpected connection with birational geometry: smooth projective varieties $X$ and $Y$ are {\em stably birational} if and only if 
\[
[X] \equiv [Y] \pmod{[\bA^1]}, 
\]
We have a surjective homomorphism
\begin{equation}
\label{eqn.toSBir}
\mathrm K_0(\Var_k) \to \bZ[\mathrm{SBir}_k]
\end{equation}
to the free abelian group on {\em stable birationality} classes of algebraic varieties over $k$, with kernel the ideal generated by $[\bA^1]$.
In particular, every motivic measure $\mu$ which is trivial on $\bA^1$ factors through the homomorphism \eqref{eqn.toSBir}.

The proof is based on 
\begin{itemize}
\item {\em Weak factorization}: every birational map can be factored into a sequence of blow-ups and blow-downs with smooth centers,
\item {\em Bittner's presentation}: the defining 
{\em Excision} relation in  $\mathrm K_0(\Var_k)$
can be replaced by the {\em Blow-up} relation: If $\widetilde{X}$ is the blow-up of a smooth projective variety $X$ in smooth $Z\subsetneq X$ and $E$ is the exceptional divisor then  
$$
[X] - [Z]= [\widetilde{X}]- [E].
$$
\end{itemize}

The paper \cite{KT} introduced the set 
$$
\Bir_{n,k}
$$
of isomorphism classes of function fields of algebraic varieties of dimension $n$ over $k$, i.e., the set of {\em birationality} classes of $n$-dimensional algebraic varieties. The free abelian group on $\Bir_{n,k}$ is
$$
\Burn_{n,k}:=\bZ[\Bir_{n,k}].
$$
The group
$$
\Burn_k:=\bigoplus_{n\ge 0} \Burn_{n,k}
$$
has a natural ring structure, induced by the product operation on algebraic varieties. There is a surjection 
$$
\Burn_k \to \mathrm{gr}(\mathrm K_0(\Var_k)),
$$
with nontrivial kernel 
\cite[Thm.\ 2.13]{borisov}.

\subsection*{Specialization of (stable) rationality}
Applications of {\em specialization} to rationality problems go back (at least) to \cite{Beau}, which established failure of rationality of certain Fano threefolds by degenerating them to conic bundles; for  these it is easier to compute the intermediate Jacobian, and if it fails to be the Jacobian of a curve, then the same holds for the generic fiber. 

The next conceptual leap was due to Voisin \cite{voisin} and Colliot-Th\'el\`ene--Pirutka \cite{CTP} who realized that specialization holds for cycle-theoretic (stable) birational invariants such as {\em integral decomposition of the diagonal}, respectively, {\em universal $\mathrm{CH}_0$-triviality}. Again, the computation of these invariants can be easier in the special fiber. 
Immediately, this led to a wealth of new results on failure of stable rationality, see, e.g., \cite{voisin}, \cite{CTP}, 
\cite{P-cyclic}, \cite{totaro-hyper}, \cite{HKTconic}, \cite{HT-Fano}, \cite{HPT-quadric}, 
\cite{HPT-quartic}, \cite{HPT-3quad}, \cite{KT-SRBS}, \cite{Schr-2}, \cite{ABP}, \cite{Schr-small-slopes}, \cite{NO-tropical}, 
as well as the surveys \cite{peyre} and \cite{CT-schrei}.

It turned out that both $\bZ[\mathrm{SBir}_k]$ and $\Burn_k$ admit specialization homomorphisms as well \cite{nicaiseshinder}, \cite{KT}. The key construction was the {\em motivic volume} homomorphism
$$
\mathrm K_0(\Var_K) \to 
\mathrm K_0(\Var_k),  \quad K=k((t)),
$$
respectively, the {\em Burnside volume} homomorphism
$$
\Burn_K\to \Burn_k. 
$$
A common refinement
$$
\mathrm K_0(\Var_K^{\dim}) \to 
\mathrm K_0(\Var_k^{\dim})
$$
of these two volume homomorphisms, with the
\emph{Grothendieck ring graded by dimension} $\mathrm K_0(\Var_k^{\dim})=\bigoplus_{d\ge 0}\mathrm K_0(\Var_k^{\le d})$,
was given in \cite{NO-refinement}.

\section{Equivariant birational geometry}
\label{sect:basic}
Given the developments outlined in Section \ref{sec.key}, it is natural to seek {\em equivariant} and other analogs of the main constructions. 
Here, we do this in the context of actions of a \emph{finite} group $G$; a version for varieties with logarithmic volume forms has been proposed in \cite{CKT}.

\subsection{Background}
\label{sect:basic1}
We follow the conventions of \cite[Sect.\ 2]{KT-vector}:
a $G$-variety is a reduced, separated, finite-type scheme over $k$, with regular action of $G$ that is transitive on the set of irreducible components.
The $G$-action is generically free if it is free on a nonempty invariant open subvariety.

Let $X$ and $Y$ be $G$-varieties.
We call a $G$-equivariant rational map $X\dashrightarrow Y$ a \emph{$G$-rational map}; if birational, then we call it \emph{$G$-birational}.
A $G$-rational map $X\dashrightarrow Y$ is \emph{proper} if there exist a $G$-variety $Z$ and proper $G$-equivariant morphisms $Z\to X$ and $Z\to Y$, with $Z\to X$ birational (cf.\ \cite[Sect.\ 2.12]{iitaka}, \cite[App.\ A]{hassetthyeon}).
Fundamental results in this setting are:
\begin{itemize}
\item \textbf{Equivariant desingularization},
which holds, e.g., by \cite{bierstonemilman}, and lets us replace any model by a smooth model.
\item \textbf{Equivariant weak factorization}
\cite{abramovichtemkin}: Any proper $G$-birational map $X\dashrightarrow Y$, where $X$ and $Y$ are smooth, can be factored as a composition finitely many $G$-birational maps given by blow-ups of invariant smooth $G$-subvarieties, and their inverses.
\item \textbf{Regular models}.
Any rational $G$-action is regular on some projective model, which by equivariant desingularization we may take to be nonsingular; see \cite{brionmodels}.
\end{itemize}

By convention, we work with regular $G$-actions.
We say that $X$ and $Y$ are \emph{$G$-birational}, written,
$$
X\sim_G Y
$$
when there exists a proper $G$-birational map $X\dashrightarrow Y$.
We say that two $G$-varieties $X$ and $Y$ are {\em stably $G$-birational} if 
$$
X\times \bP^n \sim_G Y\times \bP^m, \quad \text{some}\ n,m \in \bN,
$$
with trivial action on the second factors.
In particular, we do not assume that $X$ and $Y$ have the same dimension.

A projective $G$-variety $X$ is called {\em linearizable}, respectively {\em stably linearizable}, if there exists a $G$-representation $V$ such that 
$$
X\sim_G \bP(V), \quad \text{resp.} \quad X\times \bP^n \sim_G \bP(V)
$$
for some $n$.
If $X$ is only quasi-projective, then we apply the same terminology for the respective condition, applied to an equivariant projective model.
A $G$-action on $X$ is called {\em unirational} if there exists a $G$-equivariant dominant rational map $$
\bP(V)\dashrightarrow X.
$$

By convention, we let $G$ act on $X$ on the right, and we write $X\actsfromright G$.
Correspondingly, the $G$-action on $K=k(X)$ is on the left, $G\actsfromleft K$, where we adopt the further convention that $k(X)$ is the product of the function fields of the irreducible components of $X$.

\subsection{Cremona group}
\label{sect:cremona}

Of particular interest are varieties $X$ with {\em large} birational automorphism groups
$$
\Bir\Aut(X).
$$
There are currently no general methods to determine this group, given a variety $X$. Of course, this is easy for curves, but is a formidable problem in dimensions $\ge 2$.

There are important classes of varieties, where $\Bir\Aut(X)$ is accessible, in principle:
\begin{itemize}
\item varieties of {\em general type}, via canonical models
supplied by the Minimal Model Program \cite{HM}, 
\item K3 surfaces: every birational automorphism is a regular automorphism, and those are classified, see, e.g., \cite[Chapter 15]{Huy} or \cite{Brandhorst},
\item {\em birationally rigid varieties}, e.g., smooth quartic threefolds \cite{MI-lu}; see \cite{pukh}.
\end{itemize}

However, one of the most important examples,
$$
\Cr_n:=\Bir\Aut(\bP^n), 
$$
the {\em Cremona group} in dimension $n\ge 2$, remains mysterious.
Taking $k$ to be algebraically closed,
the classification of finite subgroups of 
$\mathrm{Cr}_2$ was a culmination of decades of efforts \cite{DI}. 
There is a classification of finite  nonabelian {\em simple} and {\em quasi-simple} subgroups of $\mathrm{Cr}_3$ \cite{ProSimple}, \cite{blancfinite}, and there are fascinating results about the overall structure of $\mathrm{Cr}_n$ \cite{pro-jordan}, \cite{blanc-quo}, but there are also many concrete unanswered questions, that fall under the general, overarching problems:
\begin{itemize}
\item When are two (necessarily linear) $G$-actions on $\bP^n$ conjugate in $\Cr_n$?
\item Is a given $G$-action on a rational variety (stably) linearizable?
\end{itemize}

\subsection{Equivariant birational rigidity}

The study of $G$-equivariant birational rigidity is a thriving area, with many beautiful constructions, see, e.g., \cite{Pro-ECM}, \cite{CS}.  
In principle, one can approach the problem of (equivariant) birationality via classification of all possible blow-ups and blow-downs. In practice, this is feasible only in small dimensions, and under the assumption that $G$ is large. 
It entails a detailed analysis of singularities, explicit invariant theory, geometric inequalities, etc. 
In dimension 2, the main technical tool is the theory of equivariant {\em Sarkisov links}; it allowed to classify finite subgroups of $\Cr_2$ in \cite{DI}. 
In dimension 3, it is the equivariant {\em Noether-Fano inequality}. Here is a sample application:   

\medskip
\noindent
{\em Example:}
The permutation $\fA_5$-action (and thus also the $\fS_5$-action) on the diagonal quadric 
$$
X=\Bigl\{ \sum_{j=1}^5 x_j^2=0 \Bigl\} \subset \bP^4
$$ 
is not linearizable \cite{CSZ}, \cite{PZ}. By \cite[Thm. 4.1]{CTZ-3}, the $\fS_5$-action is stably linearizable.

\subsection{Stacks}
\label{sect:stacks}
Let $X$ be a $G$-variety and $[X/G]$ the associated Deligne-Mumford stack.
When $X$ is a point (with trivial $G$-action), this is the classifying stack 
$$
BG=[\Spec(k)/G].
$$
A systematic introduction to the birational geometry of Deligne-Mumford stacks can be found in \cite{KT-stacks}.

While there is loss of information in the passage
$$
X\actsfromright G \quad \rightsquigarrow  \quad [X/G] \quad \rightsquigarrow  \quad  k(X)^G,
$$
there is still a rich source of obstructions to $G$-birationality to be explored in the setting of stacks; we will see such examples below in Section~\ref{sect:amitsur}. 

An equivariant sheaf (of abelian groups) on $X$ (for the $G$-action) determines a sheaf on $[X/G]$, by which we mean a sheaf on the \'etale site of $[X/G]$ \cite[Defn.\ 4.10]{DM}.
As on any site, there are cohomology groups of sheaves of abelian groups.
For $BG$, when the base field $k$ is algebraically closed, sheaf cohomology recovers group cohomology
\begin{equation}
\label{eqn.recovergroupcoho}
\rH^p(BG,\cF)=\rH^p(G,\cF(\Spec(k))),
\end{equation}
by the \v Cech spectral sequence for the covering of $BG$ by $\Spec(k)$.
Still supposing $k$ to be algebraically closed, the Leray spectral sequence,
applied to the morphism of stacks 
$$
[X/G] \to BG,
$$
yields
\begin{equation} 
\label{eqn:spectral}
\rH^p(G,\rH^q(X, \cF)) \Rightarrow \rH^{p+q}([X/G], \cF). 
\end{equation}

One can also consider an arithmetic version, where instead of a $G$-action we consider the action of the absolute Galois group $G_k=\Gal(\bar{k}/k)$. In this case, the spectral sequence takes the form
\begin{equation} 
\label{eqn:spectral-galois}
\rH^p(G_k,\rH^q(X_{\bar{k}}, \cF)) \Rightarrow \rH^{p+q}(X_k, \cF). 
\end{equation}
This spectral sequence has been studied in great detail; see, e.g., \cite[Sect.\ 1.5]{CTSansucDuke}.

\subsection{No-name lemma}
\label{sect:noname}

This surprisingly useful result in equivariant geometry is applicable to a $G$-vector bundle, whenever the $G$-action on the base is \emph{generically free}.
In one form, it asserts the existence of a $G$-birational map, over the base, from the total space of the vector bundle to a product with an affine space (with trivial $G$-action on the affine space); cf.\ \cite[Lemma 4.4]{CGR}.
Here we state a projective version.

\begin{lemm}[No-name lemma, projective bundle version]
\label{lemm:non-name}
Let $X$ be a $G$-variety, where
the $G$-action on $X$ is generically free.
Let $E\to X$ be a $G$-vector bundle of rank $n\ge 1$.
Then, letting $G$ act trivially on $\bP^{n-1}$ we have
$$
\bP(E) \sim_G X\times \bP^{n-1},
$$
compatibly with the projection maps to $X$.
\end{lemm}

\begin{proof}
By assumption,
the $G$-action is free on some dense invariant open subscheme $W\subset X$.
If $V\subset W$ is an affine dense open subscheme, then the intersection
\[
U=\bigcap_{g\in G} V\cdot g
\]
of the translates of $V$ is an affine dense invariant open subscheme of $X$.
Writing $U=\Spec(R)$, we have quotient $Z=\Spec(R^G)$ such that
$U\to Z$ has a structure of a $G$-torsor. Then, by \cite[Prop.\ 0.9, Ampfl.\ 1.3]{mumford-git} and standard faithfully flat descent, we may identify $G$-vector bundles on $U$
with vector bundles on $Z$.
In particular, when we shrink $U$ to a suitable invariant dense affine open and correspondingly shrink $Z$, we may suppose that the restriction of $E$ to $U$ is $G$-equivariantly trivial.
With this we obtain a proper $G$-birational map $\bP(E)\dashrightarrow X\times \bP^{n-1}$.
\end{proof}

As an immediate consequence, we have

\begin{coro}
\label{coro:no-name}
Let $G\to \GL(V^\vee)$ and $G\to \GL(W^\vee)$ be representations, that determine faithful
projective representations $G\to \PGL(V^\vee)$ and $G\to \PGL(W^\vee)$.
Then $\bP(V)$ and $\bP(W)$ are stably $G$-birational.
\end{coro}

Another useful observation concerns the canonical compactification $\bP(E\oplus 1)$ of a $G$-vector bundle $E\to X$.
The $G$-rational map $E\dashrightarrow \bP(E)$ becomes a morphism after
blowing up the zero-section $0_E$.
This extends to the compactification, and we have
an equivariant isomorphism
\[ B\ell_{0_E}\bP(E\oplus 1)\cong \bP(\cO_{\bP(E)}(-1)\oplus 1) \]
over $\bP(E)$.

\begin{coro}
\label{coro:no-name2}
Let $X$ be a $G$-variety and $E\to X$ a $G$-vector bundle of rank $\ge 1$.
If the $G$-action on $\bP(E)$ is generically free, then
\[ \bP(E\oplus 1)\sim_G \bP(E)\times \bP^1, \]
where $G$ acts trivially on $\bP^1$.
\end{coro}

\section{Obstructions in $G$-equivariant geometry}
\label{sect:obstr}
Here, we assume that $k$ (the base field, of characteristic zero) is algebraically closed.
A regular action of a group $G$ on a smooth projective $n$-dimensional variety $X$ naturally induces an action on its points $X(k)$ and on various geometric invariants associated with $X$. Among those geometric invariants are:
\begin{itemize}
\item spaces of differential forms,
\item Chow and Hilbert schemes,
\item Picard and Brauer groups,
\item \'etale cohomology and Chow groups,
\item Intermediate Jacobians, etc. 
\end{itemize}
Frequently, the information extracted from these data can be turned into a birational invariant of the $G$-action.
In this section, we recall some of these
extensively studied invariants and resulting obstructions to (stable) linearizability, and introduce new invariants.

\subsection{Dimension}
Let $X$ be a smooth projective rational variety with a generically free action of a finite group $G$. 
The simplest obstruction to linearizability 
of the action is the absence of faithful $G$-representations of dimension $\dim(X)+1$. 

\medskip
\noindent
{\em Examples:} Quadrics with actions of large groups. E.g., the smallest faithful representation of the symmetric group $\fS_n$ has dimension $n-1$. It follows that the $(n-3)$-dimensional quadric $X$, defined in $\bP^{n-1}$ by
$$
\sum_{j=1}^n x_j^2 = \sum_{j=1}^n x_j = 0  
$$
with the obvious permutation action of $\fS_n$, is not linearizable.
For $n=6$, the action is stably linearizable, by \cite[Thm.\ 4.1]{CTZ-3}.

Another example is supplied by the Mathieu group $G=\mathfrak M_{11}$. Its smallest faithful representations have dimension $10$. One of these is a self-dual representation $V$ admitting a nontrivial invariant quadratic form, thus a
nonlinearizable action on an $8$-dimensional quadric $X\subset \bP(V)$. 
By \cite[Prop.\ 5.3]{HT-quad}, stable linearizability of the $G$-action on $X$ is implied by stable linearizability of the action of the $2$-Sylow subgroup $G_2\subset G$, which in this case is 
the semi-dihedral group of order $16$. The restriction of $V$ to $G_2$ admits a decomposition $W\oplus 1\oplus 1$, from which we may conclude $X\sim_{G_2}\bP(W\oplus 1)$.
Thus the $G$-action on $X$ is stably linearizable.

Yet another example is the Pfaffian, and thus rational, cubic fourfold $X\subset \bP^5$ considered in \cite[Rmk. 15]{BBT}:
$$
\sum_{j=1}^6 x_j^3 = \sum_{j=1}^6 x_3 = 0,  
$$
with the action of the Frobienius group $\mathrm{AGL}_1(\mathbb F_7)$ of order $42$. Its smallest faithful representation has dimension $6$, and the group cannot act generically freely on $\bP^4$. 
As shown in \cite[Thm.\ 16]{BBT}, $X\times \bP^1$ is linearizable.

\subsection{Fixed points}
\label{sect:fix}

Actions of cyclic groups on smooth projective rational varieties always have fixed points. However, this may fail for noncyclic groups (or for actions of cyclic groups on nonrational varieties). 

\medskip
\noindent
\emph{Examples:} The faithful action of the Klein four-group $\mathfrak K_4=C_2^2$ on $\bP^1$ has no fixed points.
Also, there are no fixed points for the action of the unique $C_2\oplus C_4\subset \Cr_2$ not fixing a curve of genus $\ge 1$ and not birational to an action on $\bP^2$ or $\bP^1\times \bP^1$, see \cite[Thm.\ 5]{blanc-thesis}.  

\medskip

Clearly, 
$$
X^G\neq \emptyset \quad \Leftrightarrow\quad (X\times \bP^n)^G\neq \emptyset,
$$
with $G$ acting trivially on the second factor.

\begin{prop}[{\cite{RYessential}}]
\label{prop:point}
Let $A$ be an abelian group and 
let $X$ and $Y$ be smooth projective $A$-birational varieties. 
Then 
$$
X^A\neq \emptyset \quad\Leftrightarrow\quad Y^A\neq \emptyset.
$$
\end{prop}

\begin{exam}
Consider a blow-up of a fixed point $x\in X^A$ on a smooth projective $A$-variety. The exceptional divisor $E=\bP(\cT_{X,x})$ is the projectivization of the tangent bundle $\cT_{X,x}$, a representation of $A$. The projectivization necessarily has fixed points.  

However, for {\em nonabelian} $G$,
the existence of fixed points is {\em not} a $G$-birational invariant of smooth projective $G$-varieties.
Consider $G=\fS_3$, acting on $\bP^2$, canonical compactification of the standard $2$-dimensional representation.
When we blow up the origin, the unique fixed point, we obtain
exceptional divisor $\bP^1$, which has has no fixed points. 
\end{exam}

\medskip
\noindent
{\bf (A)}: Existence of fixed points upon restriction to abelian subgroups 
$$
A\subseteq G.
$$

Condition {\bf (A)} is a $G$-equivariant stable birational invariant of smooth projective varieties.
It holds when the $G$-action is stably linearizable, and more generally, when the $G$-action is unirational.

\subsection{Determinant}
\label{sect:det}
For the action of a finite abelian group $A$, we get a finer invariant by considering the {\em weights} of the $A$-action in the tangent space $\cT_{X,x}$ to a fixed point $x\in X^A$. Let 
$$
\beta(x):=[b_1,\ldots,b_n], \quad b_j\in A^\vee :=\Hom(A,k^\times), 
$$
be the collection of these weights, i.e., characters of $A$. Let
$$
\det(\beta(x)) \in \wedge^n(A^\vee)
$$
be the {\em determinant},
defined as the wedge product $b_1\wedge\dots\wedge b_n$ and well-defined up to sign.

\begin{prop}[{\cite{RYinvariant}}]
\label{prop:ry}
Let 
$$
\pi\colon \widetilde{X}\to X
$$ 
be an equivariant birational morphism of smooth projective $A$-varieties. 
Then for all $x\in X^A$ and
$\tilde x\in \widetilde{X}^A$ with $\pi(\tilde x)=x$ we have
$$
\det(\beta(\tilde{x})) = \pm \det(\beta(x))\in \wedge^n(A^\vee).
$$
\end{prop}

\medskip
\noindent
{\em Example:} The action of the cyclic group $C_5$ on $\bP^1$, canonical compactification of a nontrivial one-dimensional representation, falls into one $C_5$-birational equivalence class for weight $\pm 1$ and another $C_5$-birational equivalence class for weight $\pm 2$ (and we have the same description for $C_5$-isomorphism classes, because of dimension $1$).
Using Proposition \ref{prop:ry} we see that for any $n\ge 1$ and odd prime $p$ the canonical compactifications of $n$-dimensional faithful linear representations of $C_p^n$ fall into precisely $(p-1)/2$ equivariant birational equivalence classes.

\medskip
\noindent
{\bf (Det)}: Existence of $A$-fixed points, for some abelian subgroup $A\subseteq G$, with a given determinant class
$$
[\det(\beta)] \in \wedge^n(A^\vee) / \pm
$$
is a $G$-equivariant birational invariant for smooth projective $G$-varieties. 

\begin{rema}
\label{rem.Detlimitation}
The invariant {\bf (Det)} can only deliver finer information than the mere existence of $A$-fixed points in the case of $G$-varieties of dimension equal to the rank of $A$.
\end{rema}

\subsection{Cohomology}
\label{sect:coho}

One can investigate the $G$-action on the Picard group $\Pic(X)$.
A key observation is that a blow-up of a smooth $G$-stable subvariety adds a permutation module to $\Pic(X)$, and that the multiplication with projective space with trivial $G$-action adds $\Z$ to $\Pic(X)$, a trivial permutation module. It follows that $\Pic(X)$, modulo permutation modules is a stable birational invariant. This yields an invariant and obstruction to stable linearizability:

Let $X$ be a smooth projective rational variety with a generically free $G$-action. 
The similarity class 
$$
[\Pic(X)]
$$
in the set of isomorphism classes of $G$-lattices, modulo the equivalence relation given by direct sums with permutation lattices,
is a stable birational invariant.

\medskip
\noindent
{\bf (SP)}: The property that the $G$-module $\Pic(X)$ is stably permutation:
$$
[\Pic(X)]=[0].
$$

Condition {\bf (SP)} holds for stably linearizable actions.

There exist stably permutation modules that are not permutation modules; some examples are recalled in \cite[Sect.\ 1]{hoshiyamasaki}.
We are not aware of algorithms to check whether a given $G$-module is stably permutation. 
However, the next, related, invariant can be computed effectively.

\medskip
The first group cohomology
$$
\rH^1(G, \Pic(X))
$$
is a stable birational invariant.

\medskip
\noindent
{\bf (H1)}: The vanishing of $\rH^1(H, \Pic(X))$, for all subgroups $H\subseteq G$.

\medskip
If the $G$-action is stably linearizable then Condition {\bf (H1)} holds.
In some cases, the invariant $\rH^1(G,\Pic(X))$ can be determined from the stabilizer stratification of $X$ \cite{BogPro}, \cite{KT-dp}.

\medskip
\noindent
{\em Example:} Let $G=C_p$ act on a smooth projective rational surface $X$. Assume that $X^G$ contains a curve of genus $g \ge 1$; such a curve is necessarily smooth and unique. Then by \cite{BogPro}, 
$$
\rH^1(G,\Pic(X)) \cong (\bZ/p\bZ)^{2g}.
$$

\subsection{Coniveau filtrations}
\label{sect:coniveau}
Let $X$ be a smooth projective $G$-variety of dimension $n$ and $A$ a $G$-module.
As in \cite[Sect.\ 3]{KT-uni} there is the cohomology
$$
\rH^i_G(X,A):=\rH^i([X/G],A).
$$
Let us suppose that $A$ is $m$-torsion, for a positive integer $m$, and consider the subgroups:
\begin{itemize}
\item 
$\rN^1\rH^i_G(X,A)$, the group of all $\alpha \in \rH^i_G(X,A)$, such that $\alpha$ vanishes in $\rH^i_G(U,A)$ for some nonempty $G$-invariant open $U\subset X$, 
\item 
$\widetilde{\rN}^1\rH^i_G(X,A)$, consisting of classes induced via equivariant morphisms $Y\to X$, with $Y$ a smooth projective $G$-variety of dimension $n-1$.
\end{itemize}
In the definition of $\widetilde{\rN}^1\rH^i_G(X,A)$,
classes are induced via the Gysin map
$$
\rH^{i-2}([Y/G],A\otimes \mu_m^{\otimes -1})\to \rH^i([X/G],A).
$$
We have
\begin{equation} 
\label{eqn:nc}
\widetilde{\rN}^1\rH^i_G(X,A) \subseteq \rN^c\rH^i_G(X, A);
\end{equation}
see \cite{BO} for background.
Exactly as in \cite[Prop.\ 2.4]{BO} we obtain:

\begin{prop}
\label{prop:n1}
The quotient
\begin{equation}
    \label{eqn:nonv}
\rN^1\rH^i_G(X,A)/\widetilde{\rN}^1\rH^i_G(X,A) 
\end{equation}
is a stable $G$-birational invariant of $X$. 
\end{prop}

This invariant has been studied in the non-equivariant setting: the quotient \eqref{eqn:nonv} (with $G=1$) can be nontrivial, e.g., with $i=3$, $A=\bZ/2\bZ$ \cite{kameko}.

\begin{rema}
\label{rem.anotherapproach}
Benoist and Ottem \cite[Thm. 4.3]{BO} consider a variant of this construction, where for particular $G$ and $A$ there are smooth projective varieties $X$ which approximate in a certain sense the cohomology of $BG$,
for which the quotient \eqref{eqn:nonv} is nontrivial.
\end{rema}

\subsection{Intermediate Jacobians}
\label{sect:jac}
An important notion in birational geometry of threefolds over closed and nonclosed ground fields is the {\em intermediate Jacobian}. It allowed to show that a smooth cubic threefold over $\bC$ is irrational \cite{CG}, showed irrationality of many conic bundles \cite{Beau}, gave rise to examples of irrational but stably rational threefolds \cite{BCSS}. Recent developments in \cite{BW1}, \cite{HT-cycle}, \cite{HT-rat}, \cite{BW2}, \cite{KuP1}, \cite{KuP2} concerning {\em torsors} under intermediate Jacobians
led to rationality criteria for certain geometrically rational threefolds over nonclosed fields. Equivariant versions of intermediate Jacobians and cycle invariants appeared in \cite{HT-quad}. 

The following gives an analog of a birational invariant introduced in \cite[Thm.\ 3.6]{CKK}; our definition does not depend on the theory of {\em atoms} from that paper.  

\begin{theo}[{\cite[Prop.\ 5.2]{KTT}}]
\label{thm:new}
Let $X$ be a smooth projective rational threefold with a regular action of a cyclic group $G=C_p$ of prime order $p$.  
We have a decomposition of the fixed locus $X^G=\bigsqcup_{\alpha} F_{\alpha}$ into a disjoint union of smooth irreducible components. Let $C$ be a smooth projective curve of genus $g\ge 2$. 
Consider
\begin{equation} 
\label{eqn:I}
I:=-I_1-2I_2+I_3, 
\end{equation}
where 
\begin{itemize}
\item $I_1$ is the number of $F_{\alpha}$ isomorphic to $C$, 
\item $I_2$ is the number of $F_{\alpha}$ birational to $C\times \bP^1$, and 
\item $I_3$ is the number of factors of the intermediate Jacobian $\iJ(X)$ isomorphic to $\mathrm{J}(C)$, with trivial $G$-action. 
\end{itemize}
Then $I$ is an equivariant birational invariant, which vanishes when the $G$-action is linearizable. 
\end{theo}

\begin{proof}
The proof is parallel to the proof in \cite{CKK}. By equivariant weak factorization, it suffices to consider blow-ups of smooth $G$-orbits; only the following blow-ups can change the shape of the invariant: 
\begin{enumerate}
\item blow-up of a $G$-fixed curve isomorphic to $C$ not lying on a $G$-fixed surface, 
\item blow-up of a $G$-fixed curve isomorphic to $C$ on a $G$-fixed surface,
\item blow-up of a $G$-stable but not $G$-fixed curve isomorphic to $C$, 
\item blowing up a $G$-orbit of curves isomorphic to $C$. 
\end{enumerate}
Case (1) admits subcases, depending on whether or not the $G$-weights in the normal bundle of $C$ are equal. 
In the first subcase, 
$$
\Delta(I_1)=2-1,\quad \Delta(I_2)=0, \quad \Delta(I_3)=1,
$$
where $\Delta$ denotes the discrepancy.  
In the second subcase,  
$$
\Delta(I_1)=-1, \quad  \Delta(I_2)=1,\quad \Delta(I_3)=1. 
$$
In Case (2), 
$$
\Delta(I_1)=1; \quad \Delta(I_2)=0, \quad \Delta(I_3)=1.
$$
In Case (3), there are no changes in $I_j$. In Case (4), the intermediate Jacobian changes by a $G$-orbit of $\mathrm{J}(C)$; all $I_j$ remain unchanged. 

Finally, for a linear action on $\bP^3$, all terms in \eqref{eqn:I} vanish. 
\end{proof}

\section{Amitsur groups}
\label{sect:amitsur}

\subsection*{Classical Amitsur invariant}
It is well-known that the canonical line bundle on a nonsingular $G$-variety is automatically $G$-linearized. On the other hand, not every element of 
the invariant Picard group $\Pic(X)^G$ can be represented by a $G$-linearized line bundle.
These observations can be formalized, via an introduction of a cohomological invariant, the Amitsur group; 
see, e.g., \cite[App.\ A]{blancfinite}:
$$
\Am^2(X,G):=\Pic(X)^G/\{[L]\,|\,\text{$G$-linearized line bundle $L$}\}.
$$
Then, for instance,
$\Am^2(X,G)=0$ when $X=\bP(V)$ for a linear representation $V$ of $G$.
If $\Am^2(X,G)=0$, then every element of $\Pic(X)^G$ admits a lift to the group
\[ \Pic(X,G)=\Pic([X/G]) \]
of $G$-linearized line bundle $L$ on $X$.

To see that $\Am^2(X,G)$ is a stable $G$-birational invariant, 
let $X$ be a smooth projective $G$-variety and $Z\subset X$ a smooth $G$-subvariety in codimension $\ge 2$.
When we blow up $Z$ in $X$ to produce $\widetilde{X}$ we obtain $\Pic(\widetilde{X})$, differing from $\Pic(X)$ by a permutation module, generated by 
components of the exceptional divisor. The corresponding $G$-invariant class is automatically $G$-linearized. It follows that 
$$
\Am^2(\widetilde{X},G)=\Am^2(X,G). 
$$
A similar argument takes care of invariance under product with a projective space.

\subsection*{Higher Amitsur invariants}
Additional cohomological invariants
can be extracted from the Leray spectral sequence \eqref{eqn:spectral}:
\begin{equation*} 
\rH^p(G,\rH^q(X, \cF)) \Rightarrow \rH^{p+q}([X/G], \cF),
\end{equation*}
for a $G$-linearized sheaf $\cF$ on a $G$-variety $X$. 

We start by recording some basic results:
\begin{itemize} 
\item 
For $i=1$ and $\cF=\bG_m$, we have 
$$
\rH^1([X/G],\bG_m)=\Pic(X,G); 
$$
see, e.g., \cite[Sect.\ 3]{KT-dp}.
\item 
If $k$ is algebraically closed, then by \eqref{eqn.recovergroupcoho} we have
$$
\rH^i(BG, \bG_m)=\rH^i(G,k^\times),
$$
for all $i$; when $i>0$ this recovers
classical group cohomology
$\rH^i(G,k^\times)\cong \rH^{i+1}(G,\bZ)$, cf.\ \cite[\S 2.1]{KT-dp}.
\end{itemize}

Combining these results, setting $\cF=\bG_m$, and taking $X$ to be a smooth projective variety with regular $G$-action and $k$ algebraically closed, the spectral sequence \eqref{eqn:spectral} 
yields (see, e.g., \cite[Sect.\ 3]{KT-dp}):
\begin{align}
\begin{split}
\label{eqn.BrXmodG}
\qquad 0&\to  \rH^1(G,k^\times)\to \Pic(X,G)\to 
\Pic(X)^G \stackrel{\delta_2}{\lra}  \rH^2(G,k^\times)\\
&\stackrel{}{\lra} \ker(\Br([X/G])\to \Br(X))\to\rH^1(G,\Pic(X))\stackrel{\delta_3}{\lra}  \rH^3(G,k^\times), 
\end{split}
\end{align}
where 
$$
\Br(-) :=\rH^2(-, \bG_m)_{\mathrm{tors}}
$$
is the {\em Brauer group}; since $X$ is smooth, the full group $\rH^2(X,\bG_m)$ is torsion and thus is $\Br(X)$; by \cite[Prop.\ 2.5 (iii)]{ant} the same holds for $[X/G]$.
The obstruction to linearizing line bundles with classes in $\Pic(X)^G$ is cohomological:
$$
\Am^2(X,G)=\mathrm{Im}(\delta_2). 
$$
We can also consider 
$$
\Am^3(X,G):=\mathrm{Im}(\delta_3)\subset \rH^3(G,k^\times); 
$$
this is also a stable birational invariant,
by the stable birational invariance of $\rH^1(G,\Pic(X))$ and
functoriality of the spectral sequence
\cite[Sect.\ 3]{KT-dp}. 
In each of the following cases we have the vanishing of $\Am^2(X,G)$ and $\Am^3(X,G)$:
\begin{itemize}
\item When $X$ has a $G$-fixed point \cite[Sect.\ 3]{KT-dp}.
\item When the $G$-action is stably linearizable, or more generally, when the $G$-action is unirational \cite[Sect.\ 2]{KT-uni}.
\end{itemize}

All possibilities for $\Am^2$ for rational surfaces have been determined in \cite[Prop. 6.7]{blancfinite}; for $\Am^3$ this 
has been done in \cite{TZ-uni}.

\subsection*{Universal torsor obstruction}
One can extend the analysis of the spectral sequence further, using different coefficients.
We suppose that $k$ is algebraically closed.
We can consider
$$
\cF=T_{\mathrm{NS}},
$$
the N\'eron-Severi torus, viewed as a sheaf. 
For $X$ smooth, projective, and rational, 
this is the torus with character group $\mathrm{NS}(X)=\Pic(X).$
Then, the low-degree sequence reads:
$$
0\to  \rH^1(G,T_{\mathrm{NS}}(k))\to \rH^1([X/G],T_{\mathrm{NS}})\to 
\rH^1(X,T_{\mathrm{NS}})^G \stackrel{\delta_2}{\lra}  \rH^2(G,T_{\mathrm{NS}}(k))
$$
We can identify
\begin{equation} \label{eqn:ident}
\rH^1(X,T_{\mathrm{NS}}) = \End(\Pic(X)). 
\end{equation}
This has a canonical, $G$-invariant, element $1_{\Pic(X)}$. 
As explained in \cite{HTtors} and \cite[Sect. 5]{KT-uni}, 
$$
\beta(X,G):=\delta_2(1_{\Pic(X)})
$$
is the obstruction to lifting the $G$-action from $X$ to a {\em universal} torsor 

\centerline{
\xymatrix{
 & \cT_X\ar[d]^{T_{\mathrm{NS}}} \\
G \times X  \ar[r] & X
}
}
\noindent 
This gives rise to: 

\medskip
\noindent 
{\bf Condition (T)}:  $\beta(X,G)=0$.

\medskip
\noindent 
This condition 
is a stable birational invariant of smooth projective irreducible rational $G$-varieties. Its failure is an obstruction to unirationality of the $G$-action on $X$ \cite[Prop.\ 5.1]{KT-uni}.

By \cite[Thm.\ 6.1]{KT-uni}, for toric varieties, where the action is via automorphisms of the torus (as an algebraic variety), the vanishing is also sufficient for the unirationality of the action; indeed, up to strata in codimension $\ge 2$, $\cT_X$ is an affine space, and the lifted action on it is linear. A classification of such liftable actions on 3-dimensional toric varieties can be found in \cite{TZ-toric}. 

The formalism of torsors for varieties over $BG$ 
in \cite{HTtors} was 
inspired by the corresponding theory of torsors over nonclosed fields, developed by Colliot-Th\'el\`ene--Sansuc, and others. Some, but not all, features of the theory carry over. The most important difference is the absence of an analog of Hilbert's Theorem 90, used extensively in \cite{CTSansucDuke}. On the other hand, over $BG$ there are additional geometric applications, including
a wealth of stably linearizable, non-linearizable actions (via birational rigidity or Burnside invariants).

The motivation for the introduction and detailed study of universal torsors is rooted in the expectation that the birational geometry of $\cT_X$ is in some sense ``easier'' than that of the base variety $X$. 
One instance of this is the following analog of 
\cite[Thm.\ 2.1.2]{CTSansucDuke}.

\begin{prop}
\label{prop:vanishbr}
Let $X$ be a smooth projective rational variety with a regular action of a finite group $G$. Assume that $\beta(X,G)=0$, i.e., the $G$-action lifts to a universal torsor $\cT_X\to X$. 
Let $\overline{\cT}_X$ be a smooth projective $G$-equivariant compactification of $\cT_X$. Then
$\Pic(\overline{\cT}_X)$ 
is a permutation module and hence
$$
\rH^1(G,\Pic(\overline{\cT}_X))=0. 
$$
\end{prop}

\begin{proof}
We follow the proof of \cite[Thm.\ 2.1.2]{CTSansucDuke}.
Since, non-equivariantly, $\cT_X$ is the
complement of the zero-sections in a sum of line bundles,
whose classes form a basis of $\Pic(X)$,
we have $k[\cT_X]^\times =k^\times$.
Thus $\Pic(\overline{\cT}_X)$ is freely generated by the components of the boundary $\overline{\cT}_X\setminus \cT_X$.
The $G$-action is a permutation action.
\end{proof}

\begin{exam}
A natural geometric construction of new actions goes as follows: given a regular $G$-action on a smooth projective rational $X$,  we can consider the action of $\widetilde{G}:=G^n\wr \fS_n$  on $\widetilde{X}:=X^n$, combining the given action with the permutation action. Then 
$$
\beta(X,G)=0 \quad \Rightarrow \quad  \beta(\widetilde{X}, \widetilde{G})=0. 
$$
Indeed, we have 
$$
\Pic(\widetilde{X})=\Pic(X)^n; 
$$
Putting $\cT_{\widetilde{X}}:=(\cT_X)^n$, 
where $\cT_X$ be a universal torsor of $X$,
we obtain a universal torsor for $\widetilde{X}$. 
The $G$-action lifts to $\cT_X$, by assumption. Then 
$\widetilde{G}$-action lifts to $\cT_{\widetilde{X}}$. 
\end{exam}

\subsection*{Condition {\bf (T)}, torsors, and Amitsur invariants}
We continue to suppose $k$ algebraically closed and $X$ smooth, projective, and rational.
We connect the $\beta(X,G)$-obstruction with the problem of lifting $G$-actions to general torsors and the Amitsur group $\Am^2(X,G)$.

Consider a $G$-torus $S$, i.e., non-equivariantly $S$ is isomorphic to $\G_m^n$ for some $n\in \bN$, and the $G$-action is given by a representation $G\to \GL(M)$, where $M$ denotes the character lattice of $S$.
We fix a $G$-equivariant homomorphism
\[ \lambda\colon M\to \Pic(X). \]
For this torus as well there is the low-degree sequence of the spectral sequence \eqref{eqn:spectral}:
\[
0\to  \rH^1(G,S(k))\to \rH^1([X/G],S)\to 
\Hom_G(M,\Pic(X)) \to  \rH^2(G,S(k)),
\]
where we use the equivariant identification
$$
\rH^1(X,S)=\Hom(M,\Pic(X)).
$$
We interpret $\rH^1([X/G],S)$ as isomorphism classes of $G$-equivariant $S$-torsors on $X$, and for such, mapping to $\lambda\in \Hom_G(M,\Pic(X))$, we call $\lambda$ the \emph{equivariant type} of the corresponding $G$-equivariant $S$-torsor.

The low-degree sequence is related to  the sequence \eqref{eqn:spectral} by functoriality, with $1_{\Pic(X)}\in \End(\Pic(X))$
mapping to $\lambda$.
Consequently, the image of $\beta(X,G)$ in $\rH^2(G,S(k))$ is the obstruction to realizing $\lambda$ as the $G$-equivariant type of a $G$-equivariant $S$-torsor on $X$.

The case $M=\bZ$, with trivial action, corresponds to a $G$-invariant class $[L]\in \Pic(X)^G$.
Here, $\rH^2(G,S(k))=\rH^2(G,k^\times)$, and the functoriality relation supplies the commutative diagram
\[
\xymatrix{
\End_G(\Pic(X))\ar[d]\ar[r]^{\delta_2} & \rH^2(G,T_{\mathrm{NS}}(k)) \ar[d] \\
\Pic(X)^G \ar[r]^{\delta_2} & \rH^2(G,k^\times) 
}
\]
where the vertical maps are given by evaluation on $[L]$.
Consequently, $$
\delta_2([L])=\mathrm{ev}_{[L]}(\beta(X,G)),
$$
and
\[ \mathrm{Am}^2(X,G)=\{\mathrm{ev}_{[L]}(\beta(X,G))\,\, |\,\, [L]\in \Pic(X)^G\}. \]

\begin{prop}
Let $X$ be a smooth projective rational variety with regular $G$-action.
If $\beta(X,G)=0$, then
every $G$-equivariant homomorphism $\lambda\colon M\to \Pic(X)$ arises as the equivariant type of some $G$-equivariant $S$-torsor on $X$.
In particular,
$$
\beta(X,G)=0\quad \Rightarrow \quad \mathrm{Am}^2(X,G)=0.
$$
\end{prop}

Examples with \emph{nonvanishing} $\beta(X,G)$ and vanishing $\mathrm{Am}^2(X,G)$ appear in \cite{KT-uni} (one with vanishing and one with nonvanishing $\mathrm{Am}^3(X,G)$).

\subsection*{Motivic cohomology}
Let $q$ be an integer.
The $q$th \emph{motivic complex} is the bounded above cochain complex of \'etale sheaves with transfers
$$
\bZ(q):= \begin{cases}
\mathrm{Sing}^{\bA^1}(\bZ_{tr}(\G_m^{\wedge q}))[-q], & \text{if $q\ge 0$},\\
0, & \text{if $q<0$},
\end{cases}
$$
on smooth separated finite-type $k$-schemes;
see, e.g., \cite[Defn.\ 3.1]{weibel}.
Here, 
\begin{itemize} 
\item $\mathrm{Sing}^{\bA^1}$ denotes the Suslin-Voevodsky construction \cite[Defn.\ 2.14]{weibel}, which associates to a given presheaf with transfers the complex, formed by its values on products with algebraic simplices,
\item 
$\bZ_{tr}(X)$ is the presheaf with transfers sending $U$ to the free abelian group on the set of correspondences (finite over $U$) to $X$, and
\item $\G_m$ denotes the pointed scheme $(\A^1\setminus \{0\},1)$, with $q$th iterated smash product $\G_m^{\wedge q}$.
\end{itemize}

For $q=0$, we have $\bZ(0)=\mathrm{Sing}^{\bA^1}(\bZ)$, quasi-isomorphic to $\bZ$.
For $q=1$, we have $\bZ(1)[1]=\mathrm{Sing}^{\bA^1}(\bZ_{tr}(\G_m))$, quasi-isomorphic to the \'etale sheaf $\mathbb G_m$.

We are interested in \'etale motivic cohomology groups
\begin{equation}
\label{eqn.etalemotiviccohomology}
\rH^p(X,\bZ(q)):=\mathbb{H}^p_{\acute et}(X,\bZ(q)|_{X_{\acute et}}).
\end{equation}
For instance, when $k$ is algebraically closed and $X$ is a smooth projective variety, we have
\begin{align*} 
\rH^0(X,\bZ(0))& =\bZ,& \rH^0(X,\bZ(1))& = 0,\\
\rH^1(X,\bZ(0))& = 0,& \rH^1(X,\bZ(1))&  = k^\times,  \\
\rH^2(X,\bZ(0))& = \Hom(\pi_1(X),\bQ/\bZ)), & \rH^2(X,\bZ(1))& = \Pic(X), \\
&&\rH^3(X,\bZ(1))& = \Br(X).
\end{align*}
(Normal schemes have vanishing $\rH^1({-},\bZ)$, with $\rH^i({-},\bQ)$ for all $i\ge 1$ \cite[(2.1)]{deninger}, so we get the assertions in the left-hand column from the short exact sequence $0\to \bZ\to \bQ\to \bQ/\bZ\to 0$, and the right-hand column shows the
\'etale cohomology of $\G_m$ in low degrees.)

In \eqref{eqn.etalemotiviccohomology}, the hypercohomology on the right is $\rH^p(R\Gamma(X,\bZ(q)|_{X_{\acute et}}))$.
Another approach is to work in the category of \'etale motives.
This arises in a two-step construction:
\emph{localization} of the derived category $\mathbf{D}=\mathbf{D}(\mathrm{Sh}_{\acute et}(Cor_k))$ of \'etale sheaves with transfers at $\A^1$-weak equivalences to produce the category of effective \'etale motives $\mathbf{DM}_{\acute et}^{\mathrm{eff}}(k)$ \cite[Defn.\ 9.2]{weibel},
and \emph{stabilization}, which formally inverts tensoring with $\bZ(1)$ to yield the category of \'etale motives $\mathbf{DM}_{\acute et}(k)$.
The natural functor
\[ \mathbf{DM}_{\acute et}^{\mathrm{eff}}(k)\to \mathbf{DM}_{\acute et}(k) \]
is fully faithful, by
\cite[Prop.\ A.3]{huberkahn}
and \cite[Lemme 4.2]{ayoubrealisation}.

\begin{theo}
\label{thm.etalemotiviccohoDM}
Let $k$ be a field of characteristic $0$ and $X$ a smooth separated $k$-scheme of finite type.
Then
\[
\mathbb{H}^p_{\acute et}(X,\bZ(q)|_{X_{\acute et}}) \cong \Hom_{\mathbf{DM}_{\acute et}(k)}(\bZ_{tr}(X),\bZ(q)[p]).
\]
\end{theo}

For the convenience of the reader we record a proof of this standard result; cf.\ \cite[Thm.\ 4.12]{ayoub}.
As described in \cite{spaltenstein}, for unbounded complexes, notions such as $\mathrm{K}$-injective are relevant.
For example, among unbounded complexes the injective objects are those that not only are exact (by factoring $1_{A^\bullet}$
through $\mathrm{cone}(1_{A^\bullet})$) and termwise injective (exact left adjoint to the $n$th term functor), but are also $\mathrm{K}$-injective; cf.\ \cite[Exa.\ 3.2]{hoveyreptheory}.
Over a Grothendieck category
there is the \emph{injective model structure} on cochain complexes
(loc.\ cit., see also \cite{beke}):
fibrations are epimorphisms with termwise injective $\mathrm{K}$-injective kernel,
cofibrations are monomorphisms,
and weak equivalences are quasi-isomorphisms.
In particular, for complexes of sheaves of abelian groups on an essentially small site,
$R\Gamma$ is obtained by applying $\Gamma$ to
a fibrant replacement, i.e., a quasi-isomorphic
termwise injective $\mathrm{K}$-injective complex.

\begin{proof}
We use the identification of \'etale hypercohomology with hyperext:
\begin{equation}
\label{eqn.hyperext}
\mathbb{H}^p_{\acute et}(X,\cL^\bullet|_{X_{\acute et}})\cong \Hom_{\mathbf{D}}(\bZ_{tr}(X),\cL^\bullet[p]),
\end{equation}
for any complex $\cL^\bullet$ of \'etale sheaves with transfers.
This fact \cite[Exer.\ 6.25]{weibel}
is conveniently addressed using
the local projective model structure on complexes of presheaves on $Cor_k$, respectively $Sm_k$ (smooth separated finite-type $k$-schemes), respectively $X_{\acute et}$, described in \cite[\S 5.2]{choudhurygallauer}, and the
corresponding
right-lifted model structures
on respective complexes of sheaves
(cf.\ \cite[Cor.\ 2.7]{GKR}).
In each case, the fibrant objects are those that satisfy \'etale hyperdescent;
exact forgetful functors, that forget the transfers and restrict to $X_{\acute et}$,
send fibrant objects to fibrant objects.
It suffices to show that the object $\bZ(q)$ of $\mathbf{D}$ is $\A^1$-local, since by \cite[Lemma 9.19]{weibel} the $\Hom$ on the right in the statement of the theorem is then identified with
$\Hom_{\mathbf{D}}(\Z_{tr}(X),\bZ(q)[p])$,
and we conclude by \eqref{eqn.hyperext}.

We show, more generally, that for any \'etale sheaf with transfers $\cK$, the Suslin-Voevodsky construction
$\mathrm{Sing}^{\bA^1}(\cK)$ gives the $\A^1$-localization of $\cK$, meaning:
\begin{itemize}
\item the natural map $\cK\to \mathrm{Sing}^{\bA^1}(\cK)$ is an $\A^1$-weak equivalence,
\item $\mathrm{Sing}^{\bA^1}(\cK)$ is an $\A^1$-local complex of \'etale sheaves with transfers.
\end{itemize}
The first assertion holds by \cite[Lemma 9.15]{weibel}.
The second is more subtle and relies on $\mathrm{char}(k)=0$, but quickly reduces to the case that $k$ is algebraically closed.
Indeed, writing $\bar k$ for an algebraic closure and
\[ p\colon \A^1_k\to \Spec(k),\quad \bar p\colon \A^1_{\bar k}\to \Spec(\bar k),\quad
\alpha\colon \Spec(\bar k)\to \Spec(k), \]
we have generally for a complex $\cL^\bullet$ over $\mathrm{Sh}_{\acute et}(Cor_k)$, that
$\cL^\bullet$ is $\A^1$-local if and only if the natural map $\cL^\bullet\to Rp_*p^*\cL^\bullet$ is a quasi-isomorphism.
Being a quasi-isomorphism may be detected after applying $\alpha^*$.
By \cite[Lemme 4.2]{ayoubrealisation} we get an identification of $\alpha^*Rp_*p^*\cL^\bullet$ with $R\bar p_*\bar p^*\alpha^*\cL^\bullet$.
Now the reduction is clear, since essentially by definition, we have an identification
$\alpha^*\mathrm{Sing}^{\bA^1}(\cK)\cong \mathrm{Sing}^{\bA^1}(\alpha^*\cK)$.

For the rest of the proof, we suppose $k$ is algebraically closed.
Then we may apply \cite[Prop.\ 9.30]{weibel}, which tells us that a complex over $\mathrm{Sh}_{\acute et}(Cor_k)$, which has strictly $\A^1$-homotopy invariant cohomology sheaves, is $\A^1$-local.
For $n\in \bZ$,
the cohomology presheaf $\rH^n(\mathrm{Sing}^{\bA^1}(\cK))$ is
$\A^1$-homotopy invariant by \cite[Cor.\ 2.19]{weibel}.
This leads to strict $\A^1$-homotopy invariance of the sheafification $\cF:=a_{\acute et}\rH^n(\mathrm{Sing}^{\bA^1}(\cK))$, i.e., the property that
\[ R\Gamma(U,\cF)\to R\Gamma(U\times \A^1,\cF) \]
is a quasi-isomorphism
for $U$ in $Sm_k$.
Indeed, this holds if and only if the same is true with $\cF$ replaced by $\cF\otimes^{\mathbf{L}}\bZ/\ell\bZ$ for $\ell$ prime and by $\cF\otimes \bQ$, since cohomology commutes with $\otimes^{\mathbf{L}}\bZ/\ell\bZ$, and with $\otimes\bQ$, the latter an instance of commuting with filtered colimits \cite[Rmk.\ III.3.6]{milne}.
The case of $\cF\otimes^{\mathbf{L}}\bZ/\ell\bZ$ is taken care of by the Suslin rigidity theorem \cite[Thm.\ 7.20]{weibel}, and $\cF\otimes \bQ$, by \cite[Cor.\ 14.22]{weibel} and \cite[Thm.\ 5.6, Prop.\ 5.27]{voevodcohomological}.
\end{proof}

\begin{rema}
\label{rem.SuslinVoevodskyofcomplex}
By the same argument, for an arbitrary complex $\cK^\bullet$ of \'etale sheaves with transfers,
$\mathrm{Sing}^{\bA^1}(\cK^\bullet)$ gives the $\A^1$-localization of $\cK^\bullet$ \cite[Cor.\ 4.11]{ayoub}.
\end{rema}

By \cite[Prop.\ 3.3.1]{voevodtriangulated}, we have the exact
sheafification functor 
$$
\mathrm{Sh}_{Nis}(Cor_k)\to \mathrm{Sh}_{\acute et}(Cor_k),
$$
thus a triangulated functor
\[ \mathbf{DM}_{Nis}^{\mathrm{eff}}(k)\to \mathbf{DM}_{\acute et}^{\mathrm{eff}}(k). \]
We have the projective bundle and blow-up formulas in $\mathbf{DM}_{Nis}^{\mathrm{eff}}(k)$ \cite[Thm.\ 15.12, Cor.\ 15.13]{weibel}:
\begin{gather*}
\bZ_{tr}(\bP(E))\cong \bigoplus_{r=0}^{d-1} \bZ_{tr}(X)(r)[2r]\qquad (\mathrm{rk}(E)=d), \\
\bZ_{tr}(B\ell_Z(X))\cong \bZ_{tr}(X)\oplus\bigoplus_{r=1}^{d-1} \bZ_{tr}(Z)(r)[2r]\qquad (\mathrm{codim}(Z)=d).
\end{gather*}
So the same formulas are valid in $\mathbf{DM}_{\acute et}^{\mathrm{eff}}(k)$ and we conclude:

\begin{prop}
\label{prop.blowup}
For a vector bundle $E\to X$ of rank $d$
on a smooth separated finite-type scheme $X$ over $k$,
respectively, a smooth subscheme $Z\subset X$ of codimension $d$,  we have
\begin{gather*}
\rH^p(\bP(E),\bZ(q)) \cong \bigoplus_{r=0}^{d-1}\rH^{p-2r}(\bP(E),\bZ(q-r)),\\
\rH^p(B\ell_Z(X),\bZ(q)) \cong \rH^p(X,\bZ(q))\oplus \bigoplus_{r=1}^{d-1} \rH^{p-2r}(Z,\bZ(q-r)).
\end{gather*}
\end{prop}

\begin{proof}
We apply Theorem \ref{thm.etalemotiviccohoDM} and the projective bundle and blow-up formulas in $\mathbf{DM}_{\acute et}^{\mathrm{eff}}(k)$.
\end{proof}

\subsection*{Applications of motivic cohomology}
We suppose that $k$ is algebraically closed and $X$ is a smooth projective rational variety.
This implies, in particular, that the birational invariant $\Br(X)$ of $X$ vanishes.
Let a regular action of $G$ on $X$ be given.

We use scheme approximations
in order to apply a result such as Proposition \ref{prop.blowup} to $[X/G]$.
This is the construction \cite[Sect. 6.3]{EG},
used to define equivariant (higher) Chow groups.
For each $i>0$, a representation $V$ of $G$ is chosen, with invariant open $U\subset V$ and $V\setminus U$ of codimension $\ge i$, such that a $G$-principal bundle $U\to U/G$ exists in the category of smooth separated schemes of finite type over $k$.
By \cite[Prop.\ 7.1]{mumford-git}, the same holds for
\[ X\times U\to (X\times U)/G. \]
The scheme approximation
$(X\times U)/G$ has the property, that
\[ \rH^p([X/G],\Z(q))\to \rH^p((X\times U)/G,\Z(q)) \]
is an isomorphism, provided $i\ge \min((p+1)/2,q)$,
as we see by
comparing the Leray spectral sequence \eqref{eqn:spectral}
for $X$ and $X\times U$ and applying homotopy invariance and purity
\cite[Prop.\ 3.13]{KN-I}.

In the Leray spectral sequence \eqref{eqn:spectral}, we put $\mathcal F:=\bZ(1)$ and analyze the outcome. In particular, we are interested in the kernel $\ker(\varphi)$ of the natural homomorphism
$$
\varphi\colon\rH^4 ([X/G],\bZ(1)) \to \rH^4(X,\bZ(1))^G. 
$$
Considering deeper terms in the spectral sequence, we obtain an exact sequence
\begin{equation}
\label{eqn.deeper}
\rH^3(G,k^\times)\to \ker(\varphi)\to 
\rH^2(G, \Pic(X)) \stackrel{\delta_4}{\lra} \rH^4(G,k^\times),
\end{equation}
where  $\delta_4:=d_2^{22}$ is the differential of the spectral sequence. 
This allows to define the degree $4$ Amitsur group 
$$
\Am^4(X,G):=\mathrm{Im}(\delta_4) \subseteq \rH^4(G,k^\times). 
$$

\begin{theo}
\label{thm:amit4}
Let $X$ be a smooth projective rational variety, equipped with a regular action of a finite group $G$. Then 
$\Am^4(X,G)$ satisfies the following properties:
\begin{enumerate}
\item It is a stable birational invariant of the $G$-action.
\item It vanishes if $X^G\neq \emptyset$. 
\item Given a $G$-equivariant morphism $Y\to X$ of smooth projective rational
$G$-varieties one has
$$
\Am^4(X,G)\subseteq \Am^4(Y,G). 
$$
If this induces an isomorphism $\Pic(X)\cong\Pic(Y)$ then  
$$
\Am^4(X,G)=\Am^4(Y,G).
$$
\item If $V$ is a linear representation of $G$ then 
$$
\Am^4(\bP(V),G)=0. 
$$
\end{enumerate}
\end{theo}

\begin{proof}
First we address the birational invariance.    By equivariant weak factorization, we only need to consider blow-ups of smooth subvarieties.  

Let $X$ be a smooth projective variety and $Z\subset X$ a smooth subvariety of codimension $d$. Let $\widetilde{X}\to X$ be the blow-up of $X$ in $Z$. 
By Proposition \ref{prop.blowup} and the vanishing of $\bZ(q)$ when $q<0$, we have 
\begin{equation} 
\label{eqn:blow-stack}
\rH^p(\widetilde{X}, \bZ(1))=
\rH^p(X, \bZ(1)) \oplus \rH^{p-2}(Z,\bZ(0)). 
\end{equation}  
For $p=4$, the additional term takes the form
\begin{equation}
\label{eqn:z}
\rH^{2} (Z,\bZ(0))=\Hom(\pi_1(Z),\bQ/\bZ). 
\end{equation}
Same formulas hold for stacks, as can be seen via scheme approximations. 

Assume that $Z$ is a smooth $G$-subvariety of $X$, and let $Z_0$ be a component of $Z$.
We let $H$ be the subgroup of $G$ that stabilizes the component $Z_0$.
Then, $[Z/G]\cong [Z_0/H]$.
We apply the blow-up formula:
\[
\xymatrix@C=18pt{
\rH^4 ([\widetilde{X}/G],\bZ(1)) \ar[r]^{\tilde\varphi}\ar@{=}[d] & \rH^4(\widetilde{X},\bZ(1))^G\ar@{=}[d]\\
\rH^4 ([X/G],\bZ(1))  \oplus \rH^2([Z_0/H],\bZ(0)) \ar[r] & \rH^4(X,\bZ(1))^G  \oplus \rH^2(Z_0,\bZ(0))^H 
}
\]

Using the low-order terms of the Leray spectral sequence for $\cF=\bZ(0)$, we have
$$
0\to \rH^2(H,\bZ)\to \rH^2([Z_0/H],\bZ(0))\to \rH^2(Z_0,\bZ(0))^H,
$$
thus
$$
\ker(\tilde\varphi)\cong \ker(\varphi)\oplus \rH^2(H,\bZ).
$$
One the other hand,
$$
\rH^2(G,\Pic(\widetilde{X}))\cong\rH^2(G,\Pic(X))\oplus \rH^2(H,\bZ).
$$
Comparing the exact sequence \eqref{eqn.deeper} for $X$ and $\widetilde{X}$, we see that
$$
\Am^4(X,G)=\Am^4(\widetilde{X},G).
$$

We now address invariance under upon passage to the product with a projective space with trivial $G$-action. 
This uses the following formula for motivic cohomology, a trivial case of the projective bundle formula:
$$
\rH^p(X\times \bP^n, \bZ(1))=
\rH^p(X, \bZ(1)) \oplus \rH^{p-2} (X,\bZ(0)).
$$
Then we can argue as above and obtain the invariance.

Properties (2) and (3) are immediate from the functoriality of the Leray spectral sequence. Property (4) follows by considering the $G$-rational map
$$
\bP(V\oplus 1)\dashrightarrow \bP(V). 
$$
By Property (2), $\Am^4(\bP(V\oplus 1),G)=0$. 
After blowing up the origin this becomes a morphism
$$
B\ell_0\bP(V\oplus 1)\to \bP(V).
$$
By Property (1), we still have
$\Am^4(B\ell_0\bP(V\oplus 1),G)=0$; now it suffices to apply Property (3). 
\end{proof}

\begin{coro}
\label{cor:amit4}
Let $X$ be a smooth projective rational variety with regular action of a finite group $G$.
If $X$ is $G$-unirational, then $\mathrm{Am}^4(X,G)=0$.
\end{coro}


\begin{exam}
\label{exa.K4P1}
One can compute $\Am^4(X,G)$, following \cite[p. 
~84]{KT-quotient}.
We carry this out, taking $G$ to be
the Klein $4$-group, acting on $\bP^1$ by $x\mapsto -x$ and $x\mapsto x^{-1}$.
The result is
\[ \Am^4(\bP^1,G)=\rH^4(G,k^\times)\cong (\bZ/2\bZ)^2. \]
For the computation we use the invariant set of divisors $\{0,\infty\}$, and we get
$$
\rH^2(G,\Pic(\bP^1))=\rH^2(G,\bZ)\cong (\bZ/2\bZ)^2,
$$
mapping isomorphically to
$\rH^3(G,\Z\cdot(0-\infty))\cong (\bZ/2\bZ)^2$, mapping isomorphically to $\rH^4(G,k^\times)$.
\end{exam}


\subsection*{Bogomolov multipliers}
Given that many cohomological invariants vanish upon existence of fixed points, and taking into account Condition {\bf (A)}, which is necessary for equivariant unirationality and (stable) linearizability of the action, it is natural to consider generalized {\em Bogomolov multipliers}
$$
\rB^j(G,M):=\Ker\left( \rH^j(G,M)\to \bigoplus_{A} 
\rH^j(A,M)\right), 
$$
where $M$ is a $G$-module and the sum is over all abelian subgroups $A\subseteq G$, see  \cite[Sect. 2]{TZ-uni}.
In search of interesting examples, we focus on $G$ and $M$ with nontrivial $\rB^j(G,M)$, for some $j$.  

We suppose that $k$ is algebraically closed.
For $j=2$ and $M=k^\times$, with trivial $G$-action, we obtain the classical Bogomolov multiplier
$$
\rB^2(G,k^\times)= \Br_{\mathrm{nr}}(k(V)^G),
$$
the unramified Brauer group of the field of invariants of a faithful representation $V$ of $G$ \cite{B-Brauer} over an algebraically closed field $k$ of characteristic zero.

We can extend \cite[Prop.\ 2]{TZ-uni} to the degree $4$ Amitsur group.

\begin{prop}
\label{prop.Am4Bogo}
Let $X$ be a smooth projective rational variety with regular action of $G$.
If $X$ satisfies Condition \emph{\textbf{(A)}}, then
\[
\mathrm{Am}^4(X,G)\subseteq \rB^4(G).
\]
\end{prop}

\section{Equivariant birational geometry -- choice of a model}
\label{sect:model}
In this section, we fix a finite group $G$ and assume that the base field $k$ contains a primitive $e$th root of unity, where $e$ is the least common multiple of the orders of the elements of $G$.

\subsection{Stabilizer stratification}
\label{sect:stab}
Let $X$ be a smooth $G$-variety, and suppose that the $G$-action is generically free.
One can 
introduce the \emph{stabilizer poset} 
\[
\mathcal{P}(X,G):=
\{ \mathrm{Stab}(\bar x)\,|\,\bar x\in X(\bar k) \}. 
\]
the set of subgroups of $G$ (ordered by inclusion) which occur as stabilizer groups of geometric points of $X$.

We are interested in the fixed locus $X^G$ of the action, i.e., the maximal closed subscheme of $X$ on which $G$ acts trivially, and more generally in the fixed loci $X^H$ for subgroups $H\subseteq G$.
These are smooth (see, e.g., \cite[Prop.\ A.8.10]{conradgabberprasad}), and may have several components, of various dimensions.
They are collected in the \emph{stabilizer stratification} \cite{BG}:
\[
\mathcal{S}(X,G):=\{\text{components of $X^H$}\,|\,H\subseteq G\}.
\]

Any $F\in \mathcal{S}(X,G)$ gives rise to two associated subgroups of $G$:
\begin{itemize}
\item The \emph{generic stabilizer group} \[ \mathrm{I}(F)\in \mathcal{P}(X,G). \]
\item The \emph{component stabilizer group}
\[ \mathrm{D}(F):=\{g\in G\,|\,F\cdot g=F\}. \]
\end{itemize}
The generic stabilizer group is a normal subgroup of the component stabilizer group:
\[ \mathrm{I}(F)\vartriangleleft \mathrm{D}(F). \]

In case $X$ is affine, $X=\Spec(A)$, with quotient variety $\Spec(B)$, $B=A^G$, and $F=\Spec(A/\mathfrak{P})$, with prime ideal $\mathfrak{P}$ over a prime ideal $\mathfrak{p}$ of $B$, these reproduce the classical inertia and decomposition groups:
\begin{itemize}
\item $\mathrm{I}(F)=\mathrm{I}(\mathfrak{P}/\mathfrak{p})$, the inertia group of $\mathfrak{P}$ over $\mathfrak{p}$.
\item $\mathrm{D}(F)=\mathrm{D}(\mathfrak{P}/\mathfrak{p})$, the decomposition group of $\mathfrak{P}$ over $\mathfrak{p}$.
\end{itemize}
The extension of residue fields
$A_{\mathfrak{P}}/\mathfrak{P}A_{\mathfrak{P}}$ of
$B_{\mathfrak{p}}/\mathfrak{p}B_{\mathfrak{p}}$ is Galois, and the Galois group is canonically identified with $\mathrm{D}(F)/\mathrm{I}(F)$.

\begin{exam}
\label{exa.notcentral}
Let the dihedral group $\mathfrak{D}_4$ of order $8$ act on affine space $\A^3=\Spec(k[x,y,z])$, where generators $\rho$ and $\sigma$ act by
\[ (x,y,z)\mapsto (-y,x,z)\qquad\text{and}\qquad (x,y,z)\mapsto (x,-y,-z), \]
respectively.
Then the subvariety $F$, defined by $x=y=0$, belongs to $\mathcal{S}(\A^3,\mathfrak{D}_8)$, with $\mathrm{I}(F)=\langle \rho\rangle$ and $\mathrm{D}(F)=\mathfrak{D}_4$.
\end{exam}

\subsection{Models}
\label{sect:standardf}
One could try to develop a theory of birational invariants of $G$-actions by analyzing the arrangement
$\cS(X,G)$
on a given model $X$. It turns out that the theory admits significant simplifications, under some assumptions on the model, which can be achieved by performing suitable equivariant blow-ups.
We comment on the nature of the simplification in Section~\ref{sect:scifi}.

Our setting is, still, a smooth $G$-variety, where the $G$-action is generically free.
One of the first simplifications is to obtain a model $X$ for which $\mathcal{P}(X,G)$ consists of {\em abelian} subgroups of $G$. Abelianization of stabilizers has been discovered and rediscovered several times;
the earliest reference we found is \cite{bogomolov}. Subsequently, the result appeared in \cite{RYessential}, \cite{Bcan}, and \cite{BG}. 
In a more modern formulation, it appears in the framework of \emph{destackification} in the theory of algebraic stacks \cite{bergh},
\cite{berghrydh}. 
The procedure, described below in Corollary \ref{cor.divisorial}, achieves several conditions:
\begin{itemize}
\item {\bf Abelian stabilizers} --
every $H\in \mathcal{P}(X,G)$ is abelian. 
\item {\bf Divisorial form} --
the action has abelian stabilizers, and
for every $H\in \mathcal{P}(X,G)$ and $F\in \mathcal{S}(X,G)$ with $\mathrm{I}(F)=H$, putting $Y:=\mathrm{D}(F)/H$, the composite
\[
\Pic(X,G)\to \rH^1(\mathrm{D}(F), k(F)^\times )
\to \rH^1(H, k(F)^\times)^Y \to H^\vee
\]
is surjective.
The leftmost map is given by restriction to the generic point of $F$,
the middle map is the restriction map of the inflation-restriction exact sequence, and the rightmost map is
$$
\rH^1(H,k(F)^\times)^Y\subseteq \rH^1(H,k(F)^\times)=\Hom(H,k(F)^\times)\cong H^\vee.
$$
\item {\bf Standard form} -- there exists $G$-invariant open $U\subset X$, with simple normal crossing boundary, such that
\begin{itemize}
\item $G$ acts freely on $U$, 
\item for every irreducible component $Z$ of the boundary $X\setminus U$, every $g\in G$ satisfies
\[
Z\cdot g \cap Z = \emptyset \qquad \text{or} \qquad Z\cdot g=Z.
\]
\end{itemize}
\end{itemize}
In \cite{BnG},
``divisorial form'' was called ``Assumption 2''; a similar divisoriality condition, with surjection from a given group of classes of orbifold line bundles of an algebraic orbifold, appeared in \cite{Bbar}.
The condition ``standard form'' was introduced in \cite{RYessential}.

\begin{exam}
\label{exa.abeliandivisorial}
If $G$ is abelian, then the
action of $G$ on $X$ is automatically in divisorial form.
Indeed, we have $G^\vee\to \Pic(X,G)$,
and the composite with the map from the definition of divisorial form is
$G^\vee\to H^\vee$ is the restriction map, Cartier dual to the inclusion $H\to G$.
This is surjective.
\end{exam}

All three conditions are preserved under equivariant blow-ups.
Among the advantages of
divisorial form are: its formulation does not require a generically free $G$-action,
so for instance we have
Example \ref{exa.abeliandivisorial} also 
without the requirement of a generically free $G$-action;
divisorial form
is stable under general equivariant morphisms,
as we may observe by the following characterization.

\begin{prop}
\label{prop.divisorialform}
Let $X$ be a smooth $G$-variety, where the $G$-action on $X$ is generically free.
Then the action is in divisorial form if and only if the action has abelian stabilizers and at every $\bar x\in X(\bar k)$ the characters of $\mathrm{Stab}(\bar x)$, determined by the elements of $\Pic(X,G)$,
generate
$\mathrm{Stab}(\bar x)^\vee$.
\end{prop}

\begin{proof}
Given $\bar x\in X(\bar k)$, over a closed point $x\in X$ we let $H=\mathrm{Stab}(\bar x)$ and $F\in \mathcal{S}(X,G)$ be the unique component of $X^H$ containing $x$.
The action of $G$ on $X$ restricts to the trivial action of $H$ on $F$, and in particular, on the local ring $\cO_{F,x}$.
The composite map in the definition of divisorial form factors through $\rH^1(H,\cO_{F,x}^\times)$ and agrees with the composite with
\[ \rH^1(H,\cO_{F,x}^\times)\to \rH^1(H,\bar k^\times)\cong H^\vee. \]
So the condition from the definition of divisorial form, for $F$ as above, is equivalent to generation of $\mathrm{Stab}(\bar x)$ by elements coming from $\Pic(X,G)$.
\end{proof}

\begin{rema}
\label{rem.finmanylinebundles}
Given a collection of linearized line bundles $\{L_i\}_{i\in I}$ on $X$, we can replace $\Pic(X,G)$ by the subgroup generated by the classes of the $L_i$ in the definition of divisorial form, to get a
variant, \emph{divisorial form with respect to $\{L_i\}_{i\in I}$}.
This notion appears, e.g., in \cite[Sect.\ 5]{KT-struct},
and Proposition \ref{prop.divisorialform}
remains valid for this notion and subgroup.
If $X$ is in divisorial form, then $X$ is in
divisorial form with respect to some $\{L_i\}_{i\in I}$ with $I$ \emph{finite}.
This is observed in \cite[Remark 3.2]{BnG}, which also points out the following stack-theoretic formulation ($I$ finite):
$X$ is in divisorial form with respect to $\{L_i\}_{i\in I}$
if and only if
the associated morphism of stacks
\begin{equation}
\label{eqn.toproductBGm}
[X/G]\to \prod_{i\in I} B\G_m,
\end{equation}
is \emph{representable} (i.e., induces monomorphisms on geometric stabilizers).
\end{rema}

\begin{exam}
\label{exa.notdivisorial}
Divisorial form has abelian stabilizers, but not every action with abelian stabilizers is in divisorial form:
consider the 
projectivization $X$ of the $3$-dimensional irreducible representation of the alternating group $G:=\fA_4$.
There are points with stabilizer $H$ of order $4$.
We compute $\Pic(X,G)$ using exact sequence
\eqref{eqn.BrXmodG}:
the torsion is $\Hom(\fA_4,k^\times)\cong \Z/3\Z$, the free part is spanned by the class of $\cO_X(1)$ with its natural linearization.
There is no surjective homomorphism from
$\Pic([X/G])\cong \Z\oplus \Z/3\Z$
to
$H^\vee\cong (\Z/2\Z)^2$.
\end{exam}

\begin{exam}
\label{exa.notstandardform}
Standard form is divisorial, but not every divisorial action is in standard form:
consider $X:=\bP^3$ with action of
$G:=(\Z/2\Z)^4$, where respective generators send $(x:y:z:w)$ to
\[ (x:y:-z:-w),\quad
(z:w:x:y),\quad
(x:-y:z:-w),\quad
(y:x:w:z). \]
Every subgroup of $G$ of order $2$ appears as generic stabilizer along curves in $X$.
Since $G$ is abelian, the action is in divisorial form.
With
\[ Q:\ x^2+y^2+z^2+w^2=0, \]
we get $\Pic(X,G)=\langle[Q]\rangle\oplus G^\vee$, by \eqref{eqn.BrXmodG}.
For
a subgroup $H\subseteq G$, $|H|=2$, the homomorphism $\Pic(X,G)\to H^\vee$ (definition of divisorial form) maps $[Q]$ to $1\in H^\vee$ if $C\subset Q$, otherwise to $0$, for a curve $C\in \mathcal{S}(X,G)$ with $\mathrm{I}(C)=H$, and on $G^\vee$ is given by restriction.
Such $C$ can be contained in at most two boundary components of a standard form.
But exhaustive checking reveals: any subset $T\subseteq \Pic(X,G)$, generating $H^\vee$ for every subgroup $H\subseteq G$, $|H|=2$, has to have at least $3$ elements mapping to $1\in H^\vee$ for some subgroup $H\subseteq G$, $|H|=2$.
\end{exam}

\begin{prop}
\label{prop.divisorialcentral}
Let $G$ be a finite group with generically free action on a smooth variety $X$.
If the action is in divisorial form, then
$\mathrm{I}(F)$ lies in the center of $\mathrm{D}(F)$, for every
$F\in \mathcal{S}(X,G)$.
\end{prop}

\begin{proof}
With the notation of the definition of divisorial form, the action of $Y$ on $\rH^1(H,k(F)^\times)$ has to be trivial.
Since the action is given by conjugation,
$\mathrm{I}(F)$ is a central subgroup of $\mathrm{D}(F)$.
\end{proof}

\begin{exam}
\label{exa.morenotdivisorial}
In Example \ref{exa.notcentral},
$\mathrm{I}(F)$ does not lie in the center of $\mathrm{D}(F)$.
With this linear action of $\mathfrak{D}_4$ on $\A^3$, or its projective compactification $\bP^3$, we have additional examples of actions with abelian stabilizers, not in divisorial form.
\end{exam}

\begin{prop}
\label{prop.divisorial}
Let a generically free action of a finite group $G$ on a smooth algebraic variety $X$ be given, together with a simple normal crossing divisor
\[ D=D_1\cup\dots\cup D_\ell, \]
on $X$ such that each $D_i$ is smooth and $G$-invariant.
Given $\bar x\in X(\bar k)$ over $x\in X$, we define
\[ H'\subseteq H:=\mathrm{Stab}(\bar x) \]
to be the intersection of the kernels of the characters determined by the $G$-linearized line bundles $\cO_X(D_i)$ for all $i$ with $x\in D_i$, and let $d(\bar x)$ denote the dimension of the nontrivial part of the representation of $H'$ on the tangent space $\cT_{X,\bar x}$.
Then the function $d$ is identically zero if and only if the action is in standard form with boundary $D$.
If the maximal value $m$ of $d$ is positive, then
\[ \{\bar x\in X(\bar k)\,|\, d(\bar x)=m\} \]
is the set of $\bar k$-points of a smooth $G$-invariant closed subscheme $W$ of pure codimension $m$, meeting $D$ transversally, such that the blow-up $\widetilde{X}$ of $X$ along $W$ has a simple normal crossing divisor formed by the exceptional divisor and the pre-images of the $D_i$, and $d(\bar y)<m$ for all $\bar y\in \widetilde{X}(\bar k)$.
\end{prop}

\begin{proof}
We start by proving the first assertion.
For an action in standard form with boundary $D$, the generation condition of Proposition \ref{prop.divisorialform} is satisfied at $\bar x\in X(\bar k)$
over $x\in X$ with just the $G$-linearized line bundles $\cO_X(D_i)$ with $x\in D_i$,
thus $d(\bar x)=0$.
The converse follows from the observation, that for $x$ in the complement of $D$, the vanishing of $d(\bar x)$ implies the triviality of $H$.

For the remaining assertions, the condition to have geometric stabilizer group contained in $H$ defines $H$-invariant open $X^\circ\subset X$, with $x\in X^\circ$.
At any $\bar k$-point of $X^\circ$ the value of the function $d$ stays the same when we replace $G$ by $H$.
The function $d$ also remains unchanged under passage to an equivariant \'etale cover that induces bijections on geometric stabilizer groups.
So we are reduced to checking the remaining assertions under the assumption that
$G=H$ and
$x$ is the origin in a linear representation $X=V$ for some $G\to \GL(V^\vee)$,
\[ V=\chi_1\oplus\dots\oplus \chi_\ell\oplus V', \]
direct sum of $\ell$ one-dimensional representations of $G$
and a final factor $V'$, with
the $D_i$ given by the first $\ell$ factors.

The subgroup $H'$ is
the intersection of the kernels of $\chi_1$, $\dots$, $\chi_\ell$.
We write
\[ V'=(V')^{H'}\oplus V'' \]
for some subrepresentation $V''$.
Now $m=\dim(V'')$, with
$d(\bar v)\le m$ for all $\bar v\in V(\bar k)$ and equality if and only if $\bar v$ projects to $0\in V''$.
So
\[ W=\chi_1\oplus \dots\oplus \chi_\ell \oplus (V')^{H'}. \]
Blowing up $W$ in $V$ replaces $V''$ by
$B\ell_0(V'')$.
At a $\bar k$-point $\bar y$ of the exceptional divisor,
the nontrivial part of the representation on $\mathcal{T}_{B\ell_WV,\bar y}$ of the intersection of kernels of characters of $\mathrm{Stab}(\bar y)$ is contained in the tangent space to the image of $\bar y$ in $\bP(V'')$.
This has dimension $m-1$.
\end{proof}

\begin{coro}[\cite{berghrydh}]
\label{cor.divisorial}
A smooth algebraic $G$-variety, where the $G$-action is generically free, can be brought into standard form via a sequence of blow-ups at $G$-invariant smooth centers.
\end{coro}

\begin{proof}
Starting with the empty divisor, we repeatedly blow up the locus indicated in Proposition \ref{prop.divisorial}, until the function $d$ is identically zero.
\end{proof}

\begin{rema}
\label{rem.divisorial}
The function $d$ of Proposition \ref{prop.divisorial} depends only on the closed point $x\in X$ and may be extended to arbitrary points of $X$, by replacing the tangent space at $\bar x$ by the dual of the fiber at $x$ of the conormal sheaf of the closure of $\{x\}$.
Then we recover the \emph{divisorial index} of \cite{berghrydh}.
Corollary \ref{cor.divisorial} records the outcome of the \emph{divisorialification} algorithm of \cite[Thm.\ A]{berghrydh}.
\end{rema}

\section{Equivariant Burnside groups -- definitions}
\label{sect:defi}
Consider a smooth projective $G$-variety $X$ of dimension $n$, where the base field $k$ is assumed to contain enough roots of unity; the precise requirement for ``enough'' is stated at the beginning of Section~\ref{sect:model}. 
We assume that the $G$-action is generically free.
Let $\bar x$ be a $\bar k$-point of $X$, with abelian geometric stabilizer group $H\subseteq G$.
The equivalence class of the induced representation of $H$ on the tangent space $\cT_{X,\bar x}$ may be encoded by a sequence
$$
\beta(\bar x):=(b_1,\ldots, b_n) 
$$
of characters of $H$.
If $\bar x$ is an isolated point of $X^H$, then all of the characters will be nontrivial. 
In general, the multiplicity of the zero character will be the
codimension in $X$ of the component of $X^H$ containing $\bar x$.
The characters $b_1$, $\dots$, $b_n$ will generate the character group $H^\vee$.

We formulate, first, the groups of symbols that encode these actions.
At a $\bar k$-point $\bar x\in X$
with abelian geometric stabilizer group $H$, we get a sequence of characters of $H$, that is defined up to order.
At another point in the $G$-orbit of $\bar x$ the geometric stabilizer group $H\subseteq G$ gets replaced by a conjugate subgroup $H':=gHg^{-1}$,
and the characters get replaced accordingly by their $g$-conjugates.
The corresponding symbol groups are subject to order and conjugation relations.
In a next step, blow-up relations are introduced; only after introducing these relations do we obtain groups, where the equivariant birational invariants take their values.

\subsection{Symbols}
\label{ss.symbols}

Let $H$ be a finite abelian group, with character group
\[ A:=H^\vee. \]
For $n\in \N$, we introduce
\[ \cS_n(H), \]
the abelian group generated by \emph{symbols}, which take the form
of an $n$-tuple of characters
\[ \beta=(b_1,\dots,b_n),\qquad b_1,\dots,b_n\in A, \]
such that $\langle b_1,\dots,b_n\rangle=A$.
These are subject to the relation 

\noindent
\textbf{(O)} (order)
$\beta=(b_1,\dots,b_n)$ is equal to $\beta'=(b'_1,\dots,b'_n)$
if there exists a permutation $\sigma\in \fS_n$,
with $b'_i=b_{\sigma(i)}$ for $i=1$, $\dots$, $n$.

Notice that we allow $0$ to appear as a character in $\beta$.

\medskip

Let $G$ be a finite group.
The \emph{combinatorial symbols group}
\[ \mathcal{SC}_n(G) \]
is generated by \emph{symbols}
\[ (H,Y,\beta) \]
where
\begin{itemize}
\item 
$H\subseteq G$ is an abelian subgroup,
\item $Y$ is a subgroup of $Z_G(H)/H$ (with $Z_G(H)$ the centralizer of $H$ in $G$),
and
\item $\beta$ is a sequence of characters in $H^\vee$, generating $H^\vee$, of length at most $n$.
\end{itemize}
These are subject to relations 

\noindent
\textbf{(O)} $(H,Y,\beta)=(H,Y,\beta')$ if $\beta'$ is a reordering of $\beta$, 

\noindent
\textbf{(C)} (conjugation) For $g\in G$ we have $(H,Y,\beta)=(H',Y',\beta')$, where $H'=gHg^{-1}$,
$Y'=gYg^{-1}$, and by conjugation by $g$ we obtain the characters in $\beta'$ from those in $\beta$.

\medskip

Lastly, we introduce the abelian group
\[ \mathrm{Symb}_n(G), \]
where the generators are \emph{symbols}
\[
(H, Y\actsfromleft K, \beta).
\]
Here,
\begin{itemize}
\item $H\subseteq G$ is an abelian subgroup,
\item $Y\subseteq Z_G(H)/H$ is a subgroup,
\item $\beta$ is a sequence of characters in $H^\vee$, generating $H^\vee$,
\item $K$ is a finitely generated extension field of $k$, with faithful action over $k$ by $Y$, such that, if we denote by $Y_0$ the pre-image of $Y$ in $Z_G(H)$ (by the canonical homomorphism), the restriction map
\begin{equation}
\label{eqn:su}
\rH^1(Y_0,K^\times)\to \rH^1(H,K^\times)\cong H^\vee
\end{equation}
is surjective, and
\item denoting by $d$ the transcendence degree of $K$ over $k$,
the length of the sequence of characters $\beta$ is $n-d$.
\end{itemize}
These symbols are subject to relations 

\noindent
\textbf{(O)} $(H, Y\actsfromleft K, \beta)=(H, Y\actsfromleft K, \beta')$ if $\beta'$ is a reordering of $\beta$,

\noindent
\textbf{(C)} For $g\in G$ we have $(H, Y\actsfromleft K, \beta)=(H', Y'\actsfromleft K', \beta')$, where
$$
H'=gHg^{-1}, \quad Y'=gYg^{-1},
$$
$K'$ is isomorphic as a $k$-algebra to $K$ in a manner that is compatible with the
respective actions,
and $\beta'$ obtained from $\beta$ by conjugation by $g$.

\begin{rema}
\label{rem.injective}
The surjectivity condition \eqref{eqn:su} here was called ``Assumption 1'' in \cite{BnG}.
The map \eqref{eqn:su} is extraced from the
inflation-restriction exact sequence of group cohomology.
This has, further to the left, the term
$\rH^1(Y,K^\times)$, which vanishes by Hilbert's Theorem 90.
Consequently, surjectivity of \eqref{eqn:su} is equivalent
to having here an isomorphism.
\end{rema}

\subsection{Relations}
\label{ss.relations}
We define relations, mirroring the impact on the weights when a smooth $G$-variety when $X$ is blown up along a smooth invariant $G$-subvariety.
It suffices to consider the case of a blow-up in codimension $2$ (this will be explained later, in the proof of Theorem \ref{thm.biratinvarclass}); this is reflected in the form of the relations.

\

We consider the quotient 
$$
\cS_n(H) \to \cB_n(H)
$$
by the relation:

\

\textbf{(B)} (blow-up) For $\beta=(b_1,\dots,b_n)$, $n\ge 2$,
\[ 
\beta=\begin{cases}
(0,b_2,\dots,b_n),&
\text{if $b_1=b_2$}, \\
\beta_1+\beta_2,&
\text{if $b_1\ne b_2$},
\end{cases}
\]
where
\[
\beta_1:=(b_1-b_2,b_2,b_3,\dots,b_n),\qquad
\beta_2:=(b_1,b_2-b_1,b_3,\dots,b_n).
\]

\medskip

Similarly, we consider the quotient
$$
\mathcal{SC}_n(G) \to \mathcal{BC}_n(G)
$$
by the relations

\noindent
\textbf{(V)} (vanishing) $(H,Y,\beta)=0$ whenever $b_i=0$ for some $i$, 

\noindent
\textbf{(B)} if the sequence of characters $\beta=(b_1,b_2,\dots)$ has length $\ge 2$, then
$$
(H,Y,\beta)=(H,Y,\beta_1)+(H,Y,\beta_2)+(\overline{H},\overline{Y},\bar\beta),
$$
where $\beta_1:=(b_1-b_2,b_2,\dots)$, $\beta_2:=(b_1,b_2-b_1,\dots)$, $\overline{H}:=\ker(b_1-b_2)$, $\bar\beta:=(\bar b_2,\bar b_3,\dots)$ (restrictions of characters), and $\overline{Y}:=Y_0/\overline{H}$,
with $Y_0$ the pre-image of $Y$ in $Z_G(H)$, as before.

\medskip

Finally, we consider the quotient 
$$
\mathrm{Symb}_n(G) \to \Burn_n(G)
$$
by the relations 

\noindent
\textbf{(V)} $(H,Y\actsfromleft K,\beta)=0$ whenever $b_i=0$ for some $i$, 

\noindent
\textbf{(B)} for a symbol $(H,Y\actsfromleft K,\beta)$ with $n-d\ge 2$
(where $d$ denotes the transcendence degree of $K$ over $k$),
\begin{align}
\label{eqn:v}
(H,Y\actsfromleft K&,\beta)= \notag \\
&(H,Y\actsfromleft K,\beta_1)+(H,Y\actsfromleft K,\beta_2)+(\overline{H},\overline{Y}\actsfromleft \overline{K},\bar\beta),
\end{align}
with notation as above and additionally
\[
\overline{K}:=K(t),
\]
where the following recipe is carried out to produce an action of $\overline{Y}$ on $\overline{K}$.
Consider $b:=b_1-b_2$ as a primitive character of $H/\overline{H}$.
By the surjectivity of \eqref{eqn:su} and Remark \ref{rem.injective}, $b$ admits a unique lift to $\rH^1(Y_0,K^\times)$.
A choice of cocycle representation leads to a $Y_0$-action on $K(t)$,
where $\overline{H}$ acts trivially, and we get an induced action of $\overline{Y}$; see 
\cite[Sect.\ 2]{BnG}.

\begin{rema}
\label{rem.earliersymbols}
The formulation of relations \textbf{(V)} and \textbf{(B)} follows \cite[Sect.\ 2]{KT-struct}.
Earlier papers such as \cite{BnG} and \cite{KT-vector} allowed only symbols with sequences of \emph{nonzero} characters, generating $H^\vee$, with relations \textbf{(B1)}--\textbf{(B2)}, that are however equivalent to the formulation with more general symbols and relations \textbf{(V)} and \textbf{(B)}.
An additional, conventional difference in the present formulation is the requirement for $K$ to be a \emph{field}, rather than a finite product of fields, as in \cite{BnG} (formulation with Galois algebras for the group $N_G(H)/H$), or \cite{KT-struct} (with Galois algebras for subgroups of $N_G(H)/H$);
here, $N_G(H)$ denotes the normalizer of $H$ in $G$.
\end{rema}

\subsection{Class of the action} 
\label{sect:class}
Given a good model for the $G$-action, i.e., one that is in \emph{divisorial form}, as in Section~\ref{sect:standardf}, we define its class in $\cB_n$, $\BC_n$, respectively, $\Burn_n$ as follows: 
\begin{itemize}
\item For $G$ {\em abelian}, we let $\cF(X,G)$ be the set of components of the fixed locus
\[ X^G=\bigsqcup_{F\in \cF(X,G)} F \]
and, for each, record the characters of the $G$-action 
$$
\beta_{F}(X)=(b_1,\ldots, b_n), \quad b_j\in G^\vee,
$$   
in the tangent bundle at a geometric point of $F$. The class 
\[
[X\actsfromright G]:=\sum_{F\in \cF(X,G)} [k':k]\beta_{F}(X)\in \cB_n(G)
\]
with $k'$ the algebraic closure of $k$ in $k(F)$,
yields a well-defined $G$-equivariant birational invariant \cite[Thm.\ 3]{KPT}.
\item 
For arbitrary $G$, we consider the stabilizer stratification $\cS(X,G)$ introduced in Section~\ref{sect:stab}; recall that all stabilizers are abelian.  
For every component $F\in \cS(X,G)$
of dimension $d$ we record a symbol
$$
(H, Y, \beta)\qquad\text{resp.}\qquad
(H, Y\actsfromleft k(F), \beta), 
$$
where $H=\mathrm{I}(F)$ (an abelian subgroup of $G$), and the symbol records
$Y=\mathrm{D}(F)/H\subseteq Z_G(H)/H$ (containment by Proposition \ref{prop.divisorialcentral}),
respectively $Y$ with residual action,
and
$\beta=\beta_F(X)=(b_1,\ldots, b_{n-d})$ is the sequence of characters of $H$, for the action of $H$ 
in the normal bundle to $F$ in $X$ at the generic point of $F$. We record only one such symbol for the $G$-orbit of $F$. In particular, $H$, $Y$, and $\beta$ are defined only up to conjugation (reflecting the passage to a different component in the $G$-orbit of $F$); furthermore, the sequence $\beta$ is only defined up to order.
The class is
$$
[X\actsfromright G] :=\sum_{F\in \cS(X,G)/G} [k':k](H, Y, \beta_F(X))\in\BC_n(G),
$$
with $k'$ as before the algebraic closure of $k$ in $k(F)$, respectively,
$$
[X\actsfromright G] :=\sum_{F\in \cS(X,G)/G}  (H, Y\actsfromleft k(F), \beta_F(X))\in\Burn_n(G),
$$
where the sum runs over $G$-orbit representatives of the stabilizer stratification.
\end{itemize}

\begin{rema}
\label{rem.classofaction}
In parallel to the conventional differences in the definition of symbols, mentioned in Remark \ref{rem.earliersymbols}, there are alternative expressions for $[X\actsfromright G]$.
In \cite{BnG}, $[X\actsfromright G]\in \Burn_n(G)$ is expressed as a sum over conjugacy class representatives $H$ of abelian subgroups of $G$, of sums of symbols $\sum_F \mathfrak{s}_F$, where $F$ in the inner sum runs over $N_G(H)$-orbits of elements of $\mathcal{S}(X,G)$ with generic stabilizer $H$.
In the analogous expression in \cite{KT-vector}, the inner sum is over $Z_G(H)$-orbits, but this comes at the cost of a more complicated outer sum, requiring the notion of $\Pic(X,G)$-pairing on an abelian subgroup of $G$; an advantage of the shift to $Z_G(H)$-orbits was a simpler formulation of the blow-up relations.
The closest to the present formulation is \cite[p.~3027]{KT-toric}, an expression as a sum over points $x_0\in X/G$ (in the scheme sense, i.e., not just closed points).
When we have the generic point of some $F\in \mathcal{S}(X,G)/G$ over $x_0$, the contribution is $(H,Y\actsfromleft k(F),\beta_F(X))$.
Otherwise, the contribution is a symbol with $0$ in the sequence of characters, and this vanishes by relation \textbf{(V)}.
More details are provided in \cite[Sect.\ 2]{KT-struct}.
\end{rema}

\begin{theo}
\label{thm.biratinvarclass}
The recipe, for a projective $n$-dimensional $G$-variety with generically free $G$-action, to choose a smooth projective model $X$ in divisorial form and 
associate the element
\[ [X\actsfromright G], \]
defines a $G$-equivariant birational invariant class in $\BC_n(G)$, respectively in $\Burn_n(G)$.
\end{theo}

We recall, a smooth projective model in divisorial form may be obtained by performing equivariant resolution of singularities, followed by the sequence of blow-ups given by Corollary \ref{cor.divisorial}.

\begin{proof}
It needs to be checked that two different smooth projective models in divisorial form give rise to the same class in $\BC_n(G)$, respectively in $\Burn_n(G)$.
This is established in \cite[Thm.\ 5.1]{BnG} by using
equivariant weak factorization to reduce to the case of a blow-up of a smooth $G$-subvariety of codimension $j\ge 2$, and relations in $\Burn_n(G)$
of the form
\[
(H,Y\actsfromleft K,\beta)=
\sum_{\emptyset\ne I\subset \{1,\dots,j\}} 
(H_I,Y_I\actsfromleft K_I,\beta_I),
\]
provided $\mathrm{trdeg}_k(K)\le n-j$
(or analogous relations in $\BC_n(G)$);
a key point is that these relations are shown, in \cite[Prop.\ 4.7(ii)]{BnG}, to follow from relations \textbf{(B)} and \textbf{(V)}.
Here,
\[ H_I:=\bigcap_{i,i'\in I}\ker(b_i-b_{i'}) \]
and, writing $Y=Y_0/H$ with $Y_0\subseteq Z_G(H)$,
\[ Y_I:=Y_0/H_I. \]
We construct an action of $Y_I$ on the purely transcendental extension
$K_I$ of $K$, $\mathrm{trdeg}_K(K_I)=|I|-1$,
in an analogous fashion to
$\overline{Y}\actsfromleft \overline{K}$ in \eqref{eqn:v}; this is the {\em action construction} from \cite[Sect.\ 2]{BnG}, see also \cite[Sect.\ 2]{KT-vector}. 
All $b_i$ for $i\in I$ have a common class $\bar b\in (H_I)^\vee$; then,
\[ \bar \beta:=(\bar b,\bar b_{i_1},\dots, \bar b_{i_r}), \]
with $\{1,\dots,n-d\}\setminus I=\{i_1,\dots,i_r\}$.
\end{proof}

Algorithms for the computation of the class $[X\actsfromright G]\in \Burn_n(G)$ for actions arising from (projectivizations of) linear representations and actions on smooth projective toric varieties can be found in \cite{KT-vector} and \cite{KT-toric}.

\subsection{Class of an open subvariety}
\label{ss.opensubvariety}
In applications, we also need to consider a $G$-invariant open $U$ in a smooth projective $G$-variety $X$.
We suppose that the $G$-action on $X$ is generically free, and $X$ is in divisorial form.
There are \emph{two} ways to define the class of $U\actsfromright G$ (in any flavor of Burnside group).
The most naive way to do this is to copy the formula from Section \ref{sect:class}, replacing $X$ by $U$:
\begin{align*}
[U\actsfromright G]^{\mathrm{naive}}&:=\sum_{F\in \cF(U,G)} [k':k]\beta_{F}(U)\in \cB_n(G), \\
[U\actsfromright G]^{\mathrm{naive}}&:=\sum_{F\in \cS(U,G)/G} [k':k](H, Y, \beta_F(U))\in\BC_n(G), \\
[U\actsfromright G]^{\mathrm{naive}}&:=\sum_{F\in \cS(U,G)/G}  (H, Y\actsfromleft k(F), \beta_F(U))\in\Burn_n(G).
\end{align*}
These classes are birational invariants of $U\actsfromright G$, by \cite[Lemma 5.3]{BnG}.

However, motivated by the form of the specialization map, even in its original, nonequivariant form \cite{KT},
we also need a more sophisticated formula with an alternating sum over boundary strata.
For this we make the further assumption, that $X\setminus U$ is a simple normal crossing divisor \[ D=D_1\cup\dots \cup D_\ell, \]
such that each $D_i$ is $G$-invariant.
For $\emptyset\ne I\subseteq \cI:=\{1,\dots,\ell\}$ we have $D_I;=\bigcap_{i\in I}D_i$,
with normal bundle $\mathcal{N}_{D_I/X}$
and \emph{punctured normal bundle} $\mathcal{N}_{D_I/X}^\circ$.
The latter is defined by removing, from $\mathcal{N}_{D_I/X}\cong \bigoplus_{i\in I}\mathcal{N}_{D_i/X}|_{D_I}$,
the fiber over $D_{I\cup\{j\}}$, for all $j\in \cI\setminus I$, as well as the zero-section of $\mathcal{N}_{D_i/X}|_{D_I}$, for all $i\in I$.
Then in any flavor of Burnside group we have
\begin{align*}
[&U\actsfromright G]:=
[X\actsfromright G]+\sum_{\emptyset\ne I\subseteq \cI}(-1)^{|I|}[\mathcal{N}_{D_I/X}\actsfromright G]^{\mathrm{naive}} \\
&\qquad=
[U\actsfromright G]^{\mathrm{naive}}+\sum_{\emptyset\ne I\subseteq \cI}(-1)^{|I|}[\mathcal{N}_{D_I/X}^\circ\actsfromright G]^{\mathrm{naive}}.
\end{align*}
The definition is \cite[Defn.\ 5.4]{BnG}, and the equality is \cite[Lemma 5.7]{BnG}.
This class is a birational invariant of $U\actsfromright G$, by \cite[Thm.\ 5.15]{BnG}.

\subsection{Comparisons}
\label{ss.comparisons}
The different versions of Burnside groups are related to each other, via forgetful homomorphisms. The homomorphism
$$
\Burn_n(G)\to \BC_n(G)
$$
forgets the birational type of the stratum $F$ in the symbol \cite[Prop.\ 8.2]{KT-struct}: 
$$
(H, Y\actsfromleft k(F), \beta) \mapsto [k':k](H, Y, \beta).
$$
For $G$ abelian, we 
have a further homomorphism \cite[Exa.\ 8.8]{KT-struct}
$$
\BC_n(G)\to \cB_n(G), 
$$
where we quotient by the subgroup generated by symbols with $H\subsetneq G$. 
A more refined formalism of intermediate quotients, with respect to natural filtrations on $\Burn_n(G)$ is described in Section~\ref{sect:filtration}.

\section{Equivariant Burnside groups -- properties}
\label{sect:first}
We continue to assume that $k$ has enough roots of unity, as in Sections \ref{sect:model} and \ref{sect:defi}.

\subsection{Vanishing in stable range}
\label{sect:vanishing}

In the analysis of relations in the various Burnside groups, we often use the following combinatorial consequence of the blow-up relation \cite[Prop.\ 4.7(i)]{BnG}: 
a symbol
$$
(H, Y\actsfromleft K, (b_1,\ldots, b_{n-d})) 
$$
vanishes in $\Burn_n(G)$, if $\sum_{i\in I}b_i=0$ for some $\emptyset\ne I\subseteq \{1,\dots,n-d\}$.
This implies the vanishing of any symbol of the form
\[ (H, Y\actsfromleft K(t_1,\dots,t_m), (b_1,\ldots, b_{n-d})) \]
where $m$ is at least one less than the minimum of the orders of the $b_i$ \cite[Prop.\ 4.1]{KT-toric}.
As a consequence, if $X$ is a smooth projective $n$-dimensional variety with generically free $G$-action, and we consider $X\times \bP^m$, with trivial $G$-action on $\bP^m$, then
$$
[X\times \bP^m\actsfromright G]=(1, G\actsfromleft k(X)(t_1,\ldots, t_m), ())\in \Burn_{n+m}(G),
$$
provided $m\ge -1+\max_{g\in G}|\langle g\rangle|$ \cite[Rmk.\ 4.2]{KT-toric}. 
(More generally, if $X$ is a smooth projective $G$-variety and $X_0$ an irreducible component of $X$, then this holds with $(1,\mathrm{D}(X_0)\actsfromleft k(X_0)(t_1,\dots,t_m),())$ on the right.)

\subsection{Filtrations}
\label{sect:filtration}
There is the possibility to restrict attention just to certain
stabilizer groups.
This can be done systematically using the notion of filter,
introduced in \cite[\S 3]{KT-struct}.
A \emph{$G$-filter} is a subset $\mathbf{H}$ of the set of pairs $(H,Y)$, consisting of an abelian subgroup $H\subseteq G$ and a subgroup $Y\subseteq Z_G(H)/H$,
that is closed under conjugation and satisfies the following additional property.
For $(H,Y)\in \mathbf{H}$, $H\ne 1$, and $g\in Z_G(H)$, with
$\bar g\in Y$ and $Y\subseteq Z_G(g)/H$, we require
$(\langle H,g\rangle,Y/\langle \bar g\rangle)\in \mathbf{H}$.
Then
$\Burn_n^{\mathbf{H}}(G)$
is defined to be the quotient of $\Burn_n(G)$ by the subgroup generated by all symbols $(H,Y\actsfromleft K,\beta)$ with $(H,Y)\notin \mathbf{H}$.
A similar definition is made for $\BC_n(G)$.
The properties of a filter guarantee \cite[Prop.\ 3.3, Prop.\ 8.7]{KT-struct} that
$\Burn_n^{\mathbf{H}}(G)$ is the abelian group,
generated by symbols $(H,Y\actsfromleft K,\beta)$ with $(H,Y)\in \mathbf{H}$, subject to relations
\textbf{(O)}, \textbf{(C)}, \textbf{(V)}, \textbf{(B)}, applied just to these symbols,
and the analogous fact holds for $\BC_n^{\mathbf{H}}(G)$.

There is also a comparison homomorphism between the Burnside and combinatorial Burnside groups (\S \ref{ss.comparisons}) in the presence of a filter $\mathbf{H}$, and this is compatible with the quotient maps $\Burn_n(G)\to \Burn_n^{\mathbf{H}}(G)$ and
$\BC_n(G)\to \BC_n^{\mathbf{H}}(G)$.

For instance, when $G$ is abelian we define (cf.\ \cite[\S 8]{BnG})
\[ \Burn_n^G(G):=\Burn_n^{\{(G,1)\}}(G),\qquad
\BC_n^G(G):=\BC_n^{\{(G,1)\}}(G).
\]
By \cite[Exa.\ 8.8]{KT-struct}, we have
\[ \BC_n^G(G)=\cB_n(G). \]

\subsection{Incompressibles}
\label{sect:incomp}
The analysis of the blow-up relation, and in particular, the determination of the vanishing of a symbol in $\Burn_n(G)$ can be tricky.
However, in many applications, $G$-actions can be distinguished upon projection to the subgroup, freely generated by {\em incompressible} symbols. These are divisor symbols, that is, symbols 
$$
(H,Y\actsfromleft K,(b))
$$ 
with $\mathrm{trdeg}_k(K)=n-1$ and nontrivial cyclic $H$, $H^\vee=\langle b\rangle$, that cannot be obtained, via the action construction, as rightmost term in a relation \eqref{eqn:v}, coming from relation \textbf{(B)}.
The notion appears in \cite[Defn.\ 3.3]{KT-vector}.

Among classical precursors of this notion are fixed curves of birational involutions on rational surfaces, leading to the classification of involutions as Bertini, Geisser, and
de Jonqui\`eres; cf.\ \cite{baylebeauville}. A more recent version of this is the {\em normalized fixed curve with action (NFCA)} invariant, introduced by Blanc \cite{blancsubgroups}.
This takes into account the stabilizer and the residual action, and gives a finer invariant for cyclic groups of nonprime order.

\subsection{Specialization}
Let $\mathfrak{o}$ be a complete DVR with residue field $k$,
and let $K$ be the fraction field of $\mathfrak{o}$.
Fix a uniformizer $t\in \mathfrak{o}$;
so, $\mathfrak{o}\cong k[[t]]$.
We proceed to describe the  specialization map
\[ \rho_t^G\colon \Burn_{n,K}(G)\to \Burn_n(G). \]
The definition works with symbols in $\Burn_{n,K}(G)$, but the motivation is the specialization of $[X\actsfromright G]\in \Burn_{n,K}(G)$, where $X$ is a smooth projective $G$-variety over $K$ with generically free $G$-action.

We take $\cX$ to be a regular model, projective over $\mathfrak{o}$, such that the special fiber is a simple normal crossing divisor $D=D_1\cup\dots\cup D_\ell$, where the $G$-action extends to $\cX$, each $D_i$ is $G$-invariant, and there is $d_i\in \mathbb{N}$, common multiplicity of the components of $D_i$ in the special fiber.
The set-up reminds us of Section \ref{ss.opensubvariety}, and we have similarly the normal bundle $\mathcal{N}_{D_I/\cX}$ and punctured normal bundle $\mathcal{N}_{D_I/\cX}^\circ$ for $\emptyset\ne I\subseteq \cI:=\{1,\dots,\ell\}$.
We have a map, given by the canonical section of a twist of $\cO_{D_I}$, and isomorphism,
supplied by $t$:
\[
\bigotimes_{i\in I} \mathcal{N}_{D_i/\cX}^{\otimes d_i}|_{D_I}\to 
\biggl(\bigotimes_{i\in I} \mathcal{N}_{D_i/\cX}^{\otimes d_i}|_{D_I}\biggr)\otimes
\cO_{D_I}\Bigl(\sum_{j\in \cI\setminus I}d_jD_{I\cup\{j\}}\Bigr)\cong \cO_{D_I}.
\]
We define the regular function $\omega_I$ on $\mathcal{N}_{D_I/\cX}$ via the natural identification with $\bigoplus_{i\in I}\mathcal{N}_{D_i/X}|_{D_I}$,
to be the composite with the morphism
\[
\bigoplus_{i\in I}\mathcal{N}_{D_i/X}|_{D_I}\to
\bigotimes_{i\in I} \mathcal{N}_{D_i/\cX}^{\otimes d_i}|_{D_I},\qquad
(s_i)_{i\in I}\mapsto \prod_{i\in I}s_i^{d_i}.
\]
The locus of vanishing of $\omega_I$ is the union of the fibers over $D_{I\cup\{j\}}$ for $j\in \cI\setminus I$ and the zero-sections of the $\mathcal{N}_{D_i/X}|_{D_I}$.
Thus
\[
\omega_I^{-1}(\A^1\setminus\{0\})=\mathcal{N}_{D_I/\cX}^\circ.
\]

The specialization map is defined as follows.
Given a symbol
\[ (H,Y\actsfromleft K(W),\beta) \]
with $W$ a smooth projective $G$-variety over $K$, we take
$\mathcal{W}$ to be a regular model of $W$
as above.
So the $G$-action extends to
$\mathcal{W}$, projective over $\mathfrak{o}$,
with simple normal crossing special fiber
$D_1\cup\dots\cup D_\ell$, such that each $D_i$ is $G$-invariant with a multiplicity $d_i$.
There is the regular function $\omega_I$ on $\mathcal{N}_{D_I/\cW}$, as above.
Then
\[
\rho_\pi^G((H,Y\actsfromleft K(W),\beta)):=\sum_{\emptyset\ne I\subseteq \cI}(-1)^{|I|-1}(H,Y\actsfromleft k(\omega_I^{-1}(1)),\beta),
\]
which is well-defined, i.e.,
independent of the choice of $\cW$ \cite[Sect.\ 6]{BnG}.

We return to our motivation, to understand the image of $[X\actsfromright G]$ under $\rho_t^G$.
This was addressed in \cite[Thm.\ 6.6]{BnG}, with a formula as an explicit sum of symbols.
We state this in a conceptually simpler form.
\begin{theo}
\label{thm.revisitThm6.6}
Let $X$ be a smooth projective $G$-variety over $K$ with generically free $G$-action, and let $\cX$ be a regular model to which the action extends, such that the special fiber is a simple normal crossing divisor $D=D_1\cup \dots\cup D_\ell$, where each $D_i$ is $G$-invariant with a multiplicity $d_i$.
Then with $\omega_I$ as above we have
\[
\rho_t^G([X\actsfromright G])=
\sum_{\emptyset\ne I\subseteq \cI}
(-1)^{|I|-1}[\omega_I^{-1}(1)\actsfromright G]^{\mathrm{naive}}.
\]
\end{theo}

A technical point is that we wish to employ geometric arguments over $k$, but $X$ is defined over $K$.
N\'eron desingularization \cite[\S 3.6]{BLR} lets us replace $\mathfrak{o}$ by a suitable smooth $k[t]$-subalgebra $A$, and $K$ by the fraction field $K_A$ of $A$, thereby enabling direct geometric arguments.
First, $A$ may be chosen with projective $A$-scheme $X_A$, locally a complete intersection with $\cX\cong X_A\times_{\Spec(A)}\Spec(\mathfrak{o})$.
The simple normal crossing boundary translates into a local equation form for $\cX$, which we may assume valid for $X_A$.
Then $X_A$ is smooth over $k$ with simple normal crossing boundary $X_{A/tA}:=X_A\times_{\Spec(A)}\Spec(A/tA)$, and
the same holds when $A$ is replaced by a larger smooth $k[t]$-subalgebra of $\mathfrak{o}$.
By means of monoidal transform we may suppose that $X_{A/tA}$ is equivariantly isomorphic to $ D\times \Spec(A/tA)$ (where $G$ acts trivially on the second factor).

The normal crossing boundary of $X_A$ defines a morphism
\begin{equation}
\label{eqn.XAtoAlmodGml}
[X_A/G]\to [\A^\ell/\G_m^\ell],
\end{equation}
which by \cite[Rmk.\ 3.4]{berghrydh} is smooth.
The stack $[\A^\ell/\G_m^\ell]$ has a dense open substack $[(\A^1\setminus \{0\})^\ell/\G_m^\ell]\cong \Spec(k)$, with
pre-image $X_{A[t^{-1}]}=X_A\setminus X_{A/tA}$ by \eqref{eqn.XAtoAlmodGml}.
The pre-image of $X_{A[t^{-1}]}$ under
$\cX\to X_A$ is $X$.

We observe that
in Theorem \ref{thm.revisitThm6.6} there is no loss of generality in supposing $X$ to be in divisorial form.
For \eqref{eqn.XAtoAlmodGml} there is a relative notion of divisoriality and a functorial procedure to achieve this by iterated blow-up; see \cite[Thm.\ 6.6]{berghrydh}.
Each center of blow-up is smooth over $[\A^\ell/\G_m^\ell]$, i.e., smooth over $k$ and transverse to the normal crossing boundary.
We claim, also, transverse intersection on $\mathcal{N}_{D_I/\cX}^\circ$ with $\omega_I^{-1}(1)$, for every $I$.
This claim does not see the group action; with respect to a local trivialization of $\mathcal{N}_{D_I^\circ/\cX}$, we have $\omega_I$ given as unit times monomial,
so this is clear.
Thus, in Theorem \ref{thm.revisitThm6.6},
relative divisorialification preserves the hypothesis, as well as both sides of the equality
(by birational invariance of the class in the Burnside group).

The notion of divisoriality makes sense also for $\cX$, as explained in \cite[Sect.\ 6]{BnG}: $\cX$, to be divisorial, should have abelian stabilizers with $X$ divisorial, and for every $F\in \mathcal{S}(\cX,G)$ supported in the special fiber, with generic stabilizer group $H$, we should have surjective composite
\[
\Pic(\cX,G)\to 
\rH^1(\mathrm{D}(F), k(F)^\times)
\to \rH^1(H, k(F)^\times)^{\mathrm{D}(F)/H} \to H^\vee.
\]
But if $X$ is divisorial, then $\cX$ is divisorial.
The corresponding assertion for orbifolds is \cite[Lemma 9.1(ii)]{KT-stacks}; we give a proof in the present setting.
Say $X$ is divisorial with respect to linearized line bundles $L_1$, $\dots$, $L_r$.
We suppose $A$, as above, is taken so that the linearized line bundles are restrictions of linearized line bundles $M_1$, $\dots$, $M_r$ on $X_A$.
We apply relative divisorialification
to the map
\[
[X_A/G]\to B\G_m^r\times [\A^\ell/\G_m^\ell]
\]
that combines \eqref{eqn.toproductBGm} and \eqref{eqn.XAtoAlmodGml}.
A numerical quantity that measures the
relative stabilizers, if not identically zero, takes its maximal value on a locus $W$, smooth over $B\G_m^r\times [\A^\ell/\G_m^\ell]$, in particular, transverse to the normal crossing boundary.
By the equivariant isomorphism of
$X_{A/tA}$ with $D\times \Spec(A/tA)$ and the analysis of the numerical quantity in terms of representation theory (cf.\  proof of Proposition \ref{prop.divisorial}),
every component of $W$ is either contained in $X_{A[t^{-1}]}$ or meets every fiber of $X_{A/tA}$ nontrivially.
The image in $\Spec(A)$ of a component of $W$ cannot contain the generic point, but by transversality cannot be equal to $\Spec(A/tA)$.
Thus $W$ is contained in $X_{A[t^{-1}]}$, with $W\times_{X_{A[t^{-1}]}}X=\emptyset$.

The proof of Theorem \ref{thm.revisitThm6.6} will be helped by a more convenient description of $\omega_I$.
The monomial map
\[ [\A^\ell/\G_m^\ell]\to [\A^1/\G_m] \]
with exponents $d_1$, $\dots$, $d_\ell$,
when composed with \eqref{eqn.XAtoAlmodGml},
is the map given by the boundary divisor $X_{A/tA}\subset X_A$.
The chosen uniformizer $t$ supplies a lift of the composite map to $\A^1$ and a corresponding lift of \eqref{eqn.XAtoAlmodGml} to
\begin{equation}
\label{eqn.XAtoAlmodK}
[X_A/G]\to [\A^\ell/\cK],\qquad
\cK:=\ker(\G_m^\ell\to \G_m).
\end{equation}
Away from the boundary, \eqref{eqn.XAtoAlmodK} restricts to
\[ [X_{A[t^{-1}]}/G]\to [(\A^1\setminus \{0\})^\ell/\cK]\cong \A^1\setminus \{0\}, \]
given by $t$.
The morphism \eqref{eqn.XAtoAlmodK} is smooth.
Since formation of the normal bundle commutes with flat base change,
we may describe $\mathcal{N}_{D_I/\cX}$ by base change from $[\A^\ell/\cK]$,
where we have $D_i$ defined by $x_i=0$, and
$D_I$, by $x_i=0$ for all $i\in I$.
Then $\mathcal{N}_{D_I/\A^\ell}\cong \A^\ell$, but now $x_i$, for $i\in I$, is a normal bundle coordinate;
$\omega_I$ is given by $x_1^{d_1}\dots x_\ell^{d_\ell}$.
Summarizing, with $0_I\colon \A^\ell\to \A^\ell$ given by $(x_i)_{i\in \cI}\mapsto (\mathbbm{1}_{\cI\setminus I}(i)x_i)_{i\in \cI}$,
we have the diagram with fiber square
\begin{equation}
\begin{split}
\label{eqn.NDIX}
\xymatrix@C=36pt{
\mathcal{N}_{D_I/\cX} \ar[r]\ar[d] & [\A^\ell/\cK] \ar[r]^(0.58){x_1^{d_1}\cdots x_\ell^{d_\ell}} \ar[d]^{0_I} & \A^1 \\
\cX \ar[r] & [\A^\ell/\cK] \ar[r]^(0.58){x_1^{d_1}\cdots x_\ell^{d_\ell}} & \A^1
}
\end{split}
\end{equation}
where the composite top map is $\omega_I$, and the composite bottom map is $t$.

In the proof of Theorem \ref{thm.revisitThm6.6} we suppress the explicit mention of $X_A$, but its use is implicit in application of facts, valid for smooth morphisms, to a morphism such as $[\cX/G]\to [\A^\ell/\cK]$, where the justification will use the smooth morphism \eqref{eqn.XAtoAlmodK}.

\begin{proof}[Proof of Theorem \ref{thm.revisitThm6.6}]
We use the description of $\mathcal{N}_{D_I/\cX}$, and $\omega_I$, from the diagram \eqref{eqn.NDIX}.
Restriction to $\A^1\setminus \{0\}$ in the top row has the effect of replacing
$\mathcal{N}_{D_I/\cX}$ by $\mathcal{N}_{D_I/\cX}^\circ$ and making the right-hand map an isomorphism.
Stabilizer components in a smooth family are smooth, so
\begin{align}
\begin{split}
\label{eqn.EequalsF}
\sum_{E\in \cS(\mathcal{N}^\circ_{D_I/\cX},G)/G}(&H,Y\actsfromleft k(E\cap \omega_I^{-1}(1)),\beta_E(\mathcal{N}^\circ_{D_I/\cX})) \\
&=\sum_{F\in \cS(\omega_I^{-1}(1),G)/G}(H,Y\actsfromleft k(F),\beta_F(\omega_I^{-1}(1))).
\end{split}
\end{align}

We bridge the gap between $\rho_t^G([X\actsfromright G])$, a sum over
$\cS(X,G)/G$, and the sum over
$\cS(\mathcal{N}^\circ_{D_I/\cX},G)/G$ in \eqref{eqn.EequalsF}, by means of deformation to the normal bundle.
Classically, for a smooth scheme $X$ with a smooth closed subscheme $V$, the construction
\[ \cM_{V/X}^\circ:=B\ell_{V\times\{0\}}(X\times \A^1)\setminus B\ell_{V\times\{0\}}(X\times \{0\}) \]
yields a scheme with morphisms to $X$ and to $\A^1$, such that:
\begin{itemize}
\item The morphism to $\A^1$ is smooth.
\item The fiber over $0$ is $\mathcal{N}_{V/X}$, projecting to $V\subset X$.
\item The fiber over $\A^1\setminus \{0\}$ is $X\times (\A^1\setminus \{0\})$.
\end{itemize}
(Usually $V$ and $X$ are not assumed smooth;
the general case is called deformation to the normal cone.
In the first property ``smooth'' is replaced by ``flat'', and in the second, instead of the normal bundle we have the normal cone.)
We are interested in $D_I$ in $\cX$, but by flatness $\cM_{D_I/\cX}^\circ$ is obtained by base change from $D_I$ in $[\A^\ell/\cK]$.
In the latter case the deformation to the normal bundle yields $[\A^\ell/\cK]\times \A^1$.
Denoting by $u$ the coordinate on the factor $\A^1$, we have morphisms $((x_i)_{i\in \cI},u)\mapsto (u^{\mathbbm{1}_I(i)}x_i)_{i\in \cI}$ to $[\A^\ell/\cK]$ and projection to $\A^1$.

There is the additional morphism $((x_i)_{i\in \cI},u)\mapsto x_1^{d_1}\cdots x_\ell^{d_\ell}$ to $\A^1$.
The additional morphism restricts to $\omega_I$ on the fiber over $0$; we denote it by $\tilde\omega_I$.
The morphism $\tilde\omega_I$ is flat, and the restriction over $\A^1\setminus \{0\}$ is smooth.
Upon restriction we get smooth $(\tilde\omega_I,\mathrm{pr}_2)$ to $(\A^1\setminus \{0\})\times \A^1$, and
subject to the caveat mentioned just before the start of the proof, the same can be asserted
when we map the corresponding open subscheme
of $\cM_{D_I/\cX}^\circ$
to $(\A^1\setminus \{0\})\times \A^1$.
The fiber over $(\A^1\setminus \{0\})\times (\A^1\setminus \{0\})$ is isomorphic to $X\times (\A^1\setminus \{0\})$, and the fiber over $(\A^1\setminus \{0\})\times \{0\}$ is $\cN_{D_I/\cX}^\circ$.

Thus, we have a $G$-equivariant map $\cS(\mathcal{N}_{D_I/\cX}^\circ,G)\to \cS(X,G)$ 
and for $E\in \cS(\mathcal{N}_{D_I/\cX}^\circ,G)$ mapping to
$W\in \cS(X,G)$ we have
$\beta_E(\mathcal{N}_{D_I/\cX}^\circ)=\beta_W(X)$.
For $W\in \cS(X,G)$ the closure $\cW$ in $\cX$ is a regular model of $W$, with normal crossing special fiber $(D_1\cap \cW)\cup\dots\cup (D_\ell\cap \cW)$ and the same multiplicities $d_i$ as for the special fiber $D_1\cup\dots\cup D_\ell$ of $\cX$.
So
\begin{align*}
\rho^G_t&([X\actsfromright G])\\
&=\sum_{\emptyset\ne I\subseteq\cI}(-1)^{|I|-1}\sum_{W\in \cS(X,G)/G}
(H,Y\actsfromleft k(\cW\times_{\cX}\omega_I^{-1}(1)),\beta_W(X))\\
&=\sum_{\emptyset\ne I\subseteq\cI}(-1)^{|I|-1}\sum_{E\in\cS(\mathcal{N}_{D_I/\cX}^\circ,G)/G}(H,Y\actsfromleft k(E\cap \omega_I^{-1}(1)),\beta_E(\mathcal{N}_{D_I/\cX}^\circ)),
\end{align*}
and we conclude by \eqref{eqn.EequalsF}.
\end{proof}

\section{Applications}
\label{sect:apply}

We suppose that $k$ is algebraically closed. 
Our main interest is the study of finite subgroups of the Cremona group $\Cr_n$, i.e., automorphisms of rational varieties, see Section~\ref{sect:basic1}.
Representative examples of rational varieties include:
\begin{itemize}
\item quadric hypersurfaces and quadric bundles over $\bP^1$, 
\item del Pezzo surfaces and rational Fano threefolds,  
\item rational conic bundles over rational surfaces, 
\item singular cubic hypersurfaces (other than cones),
\item toric varieties,
\item equivariant compactifications of linear algebraic groups, 
\item generalized flag varieties. 
\end{itemize}
One of the main problems is the {\em linearization problem}, i.e., the determination whether or not a given action is equivariantly birational to the projectivization of a representation. This is settled in dimension 2 \cite{pinardin}, but is largely open in dimensions $\ge 3$. 

Another important problem is to prove that 
a given $G$-action on a rational variety is 
{\em stably linearizable}. The main tool for this is Lemma~\ref{lemm:non-name} (No-name lemma), applied to various vector bundle constructions. This is particularly interesting when the action is not linearizable.

In this section, we discuss several geometric applications 
of Burnside groups, quotient stacks, and cohomological invariants, with special regard to
\begin{itemize}
\item stabilizer stratification, Amitsur and Brauer groups,
\item filtrations, incompressibles,
\item specialization.
\end{itemize}

\subsection*{Fields of invariants}
An application that we have already seen, of Lemma ~\ref{lemm:non-name} (No-name lemma), is the stable $G$-birationality of projectivized representations from Corollary \ref{coro:no-name}.
By essentially the same argument, the quotients $\bP(V)/G$ and $\bP(W)/G$ are stably birational.
However, the corresponding questions for ($G$-)birationality are more subtle.

First examples of nonbirational linear $G$-actions appeared in \cite{RYinvariant}; these were based on 
Condition {\bf (Det)} in Section~\ref{sect:det}. Further examples, via the Burnside group formalism, can be found 
in \cite{CTZ-p2}   and \cite{tschinkelyangzhang}.

We are not aware of examples of 
nonbirational {\em quotients} $\bP(V)/G$, $\bP(W)/G$, for projectivized representations of the same dimension, where the $G$-actions are generically free.
The same questions can be asked for nonlinear actions, e.g., actions on quadrics. 

\subsection*{Translation actions}

Let $\mathsf G$ be a connected linear algebraic group over $k$; as a variety, $\mathsf G$ is rational. Let $G\subset \mathsf G$ be a finite subgroup. 
When is the natural translation action of $G$ on $\mathsf G$ (stably) linearizable?

The action is stably linearizable, when $\mathsf G$ is {\em special}, i.e., 
$$
\rH^1(K,\mathsf G)=1,\quad \forall\, K/k; 
$$
see, e.g., \cite[Prop. 4.1]{HT-quad}. 

Here we show how to address the stable linearizability problem for the {\em nonspecial} group $$
\mathsf G=\PGL_{n}, \quad n\ge 2.
$$
Let $V$ be a faithful representation of $G$,
and $K:=k(V)^G$ the field of invariants.
We have, by functoriality:
\begin{equation}
\label{eqn.toH1KsfG}
\rH^1(BG,\mathsf G) \to \rH^1([V/G], \mathsf G)\to \rH^1(K,\mathsf G).
\end{equation}
The inclusion of $G$ in $\mathsf G$ supplies an element $\alpha\in \rH^1(BG,\mathsf G)$.

\begin{lemm}
\label{lem.indepV}
If $\alpha$ has trivial image in $\rH^1(K,\mathsf G)$, then $\alpha$ also has trivial image in $\rH^1(k(W)^G,\mathsf G)$ for any other faithful representation $W$ of $G$.
\end{lemm}

\begin{proof}
Let $L:=k(W)^G$, and let $\beta$ denote the
image of $\alpha$ in $\rH^1(L,\mathsf G)$.
By the No-name lemma, we have on a suitable invariant open $W^\circ\subset W$ an equivariant identification $V\times W^\circ\cong \A^d\times W^\circ$, $d=\dim(V)$.
Since $\alpha$ is trivial in $\rH^1(K,\mathsf G)$ it is also trivial on a dense open subset of $[V/G]$, and this determines a dense open subset of $\A^d_L$ on which $\beta$ becomes trivial.
Since $L$-points in $\A^d_L$ are Zariski dense, we obtain the vanishing of $\beta$.
\end{proof}

We define
\[ \alpha^{\mathsf G}(G,V)\in \rH^1(k(V)^G,\mathsf G) \]
to be the image of $\alpha$ under \eqref{eqn.toH1KsfG}.
By Lemma \ref{lem.indepV},
the triviality of $\alpha^{\mathsf G}(G,V)$ is independent of the choice of faithful representation $V$.

The argument in the proof of \cite[Prop.\ 4.1]{HT-quad} implies: 

\begin{prop}
\label{prop.asinHTquad}
If $\alpha^{\mathsf G}(G,V)$ is trivial, then $\mathsf G$ is $G$-equivariantly stably linearizable.
\end{prop}

In fact, the triviality of $\alpha^{\mathsf G}(G,V)$ is both necessary and sufficient for the stable linearizability of the translation action, and indeed is encoded in the classical obstruction to lifting $G$ to a \emph{linear} representation
\[ G\to \GL_n. \]
There is a natural $\PGL_n$-equivariant compactification
$$
\PGL_n\subset \bP^{n^2-1},
$$
the projectivization of the coordinates of an $n\times n$ matrix. 
Those with a single nonzero column form an invariant linear subspace
$\bP^{n-1}\subset \bP^{n^2-1}$.
Since the inclusion induces an isomorphism on Picard groups, by \cite[Lemma 2.1]{KT-quotient}
we have $\Am^2(\bP^{n-1},G)=\Am^2(\bP^{n^2-1},G)$.
But $\Am^2(\bP^{n-1},G)$ is cyclic, generated by a class $\gamma\in \rH^2(G,k^\times)$ that represents the obstruction to lifting $G\to \PGL_n$
to a linear representation.
If $\gamma=0$,
then such a lift exists, and
by functoriality $\alpha^{\mathsf G}(G,V)$ lies in the image of
\[ \rH^1(K,\GL_n)\to \rH^1(K,\PGL_n), \]
hence is trivial, and
by Proposition \ref{prop.asinHTquad}, $\PGL_n$ is $G$-equivariantly stably linear.
If $\gamma\ne 0$, then we have nontrivial $\Am^2(\bP^{n^2-1},G)$, and this obstructs not only
the $G$-equivariant stable linearizability of $\PGL_n$ but also the $G$-equivariant unirationality.

\subsection*{Condition (A)}
Actions of abelian groups on del Pezzo surfaces have been classified in \cite{blanc-thesis}, in particular, there is a classification of all actions satisfying Condition {\bf (A)}. Such a classification for smooth Fano threefolds can be found in \cite{abban}.

\subsection*{Amitsur and Brauer groups}

The Amitsur invariant $\Am^2(X,G)$ from \cite[Sect.\ 6]{blancfinite} has been computed for rational surfaces \cite[Prop. 6.7]{blancfinite} and can be extracted for smooth Fano threefolds
from the analysis in  \cite{sharma}. 
Computations of $\Am^3(X,G)$ for del Pezzo surfaces can be found in \cite{TZ-uni}.
The related Condition {\bf (T)} (Section \ref{sect:amitsur}) has been investigated for Fano threefolds and toric varieties in \cite{KT-uni}, \cite{TZ-toric}. It would be desirable to undertake a systematic study of $\Am^4(X,G)$. 

An algorithm for computing Brauer groups of stacks $[X/G]$ and smooth projective models of quotient spaces $X/G$ has been presented in \cite{KT-quotient}.

\subsection*{Incompressibles}

The description of incompressible symbols is easy in $\Burn_2(G)$.
There, the incompressible symbols are the divisor symbols 
$(H,Y\actsfromleft k(C),(b))$,
where $C$ is either a curve of positive genus, or $C\cong \bP^1$ is acted on by a \emph{noncyclic} group $Y$.
Already the simplest example
$$
(C_2, \fS_3\actsfromleft k(\bP^1),(1)), 
$$
is useful, e.g., reproving a result of Iskovskikh \cite{isk-s3} about nonbirationality of a linear action and a toric action of $G=C_2\times \fS_3$, see \cite[\S 7.6]{HKTsmall}.

For $\Burn_3(G)$ the situation is more involved.
There are many examples of incompressible symbols, but a full classification is unavailable; it would require the detailed analysis of \cite{DI}. 
Proofs of nonlinerizability of $G$-actions on rational varieties of dimension $\ge 3$, based on incompressibles, can be found for 
\begin{itemize}
\item actions on singular cubic threefolds, e.g., \cite[Exa.\ 2.7]{CTZ-forum}, or
\item the Burkhardt quartic, \cite[Prop.\ 7.2]{CTZ-burk}.
\end{itemize}
Among further applications of incompressible symbols to proofs of nonbirationality of stably birational actions are:
\begin{itemize}
\item on $\bP^3$, see \cite[Thm.\ 11.2]{KT-vector}, or \cite[Exa.\ 8.1]{TYZ-survey}, 
\item 
cubic threefolds and degree 14 Fano threefolds \cite{TZ-14}.
\end{itemize}

\subsection*{Specialization}

An immediate application of Theorem~\ref{thm.revisitThm6.6} is the fact very general members of families of $G$-varieties equivariantly specializing to (mildly singular) nonlinearizable $G$-varieties are also not linearizable \cite[Cor.\ 6.8, Exa.\ 6.9]{BnG}.

This motivated a search for such families. For example, a systematic study of specialization patterns for singular cubic threefolds to cubics with cohomological obstructions to stable linearizability was offered in \cite{CTZ-forum}, \cite{CMTZ}. In particular, there is now a classification of actions on singular cubics $X\subset\bP^4$ failing Condition {\bf (H1)};  these come from specific $C_2$ or $C_3$-actions, see \cite[Sect.\ 2]{CTZ-real}. A representative example of such a specialization is given in \cite[Prop.~4.3]{CTZ-forum}.

\section{What if \dots}
\label{sect:scifi}

\subsection*{Nonabelian stabilizers}
\label{ss.nonabelian}
{\em A priori}, one could try to encode a given regular action of a finite group $G$ by the stabilizer stratification and the induced action in the normal 
bundles, which would decompose into irreducible representations of the stabilizer groups. There are at least two reasons for bypassing this, via passage to divisorial forms: 
\begin{itemize}
\item one would have to keep track of these representations, and it is easier to deal with sums of one-dimensional representations,
\item encoding the interaction of the action of the stabilizers with the residual action on the stratum might involve nonabelian second cohomology \cite{giraud}, which we wanted to avoid. 
\end{itemize}
This explains the importance of achieving abelian stabilizers.

\subsection*{Not in divisorial form}
\label{ss.notdivisorial}
The definition of divisorial form ensures that the formula for the class in the equivariant Burnside group yields a sum of symbols.

\begin{exam}
\label{exa.D8action}
Consider the action of $\mathfrak{D}_4$ on $\bP^3_{\C}$, the projective compactification of the action of Example \ref{exa.notcentral}.
The subvariety $F$ defined by $x=y=0$ has generic stabilizer $H$, cyclic of order $4$, and nontrivial residual action of $\mathfrak{D}_4/H$ on $\C(F)$.
As observed in \cite[Exa.\ 3.5]{BnG}, primitive characters of $H$ do not lift to $\rH^1(\mathfrak{D}_4,\C(F)^\times)$.
\end{exam}

Such lifts of characters are needed in the definition of relation \textbf{(B)}, which is crucial for the equivariant Burnside group.
The definition of symbols (\S \ref{ss.symbols}) restricts to actions $Y\actsfromleft K$, where $Y$ is quotient by $H$ of a subgroup $Y_0$ of $G$, \emph{in which $H$ is central}.
Only then is there the possibility for all characters of $H$ to lift to $\rH^1(Y_0,K^\times)$
(see Proposition \ref{prop.divisorialcentral} and \cite[Lemma 2.1]{KT-struct}).

\subsection*{Tori}
\label{ss.tori}
We suppose that $k$ is algebraically closed.
In principle, the theory developed above can be extended to the case of an algebraic torus $T=\bG_m^d$. 
We can imitate the construction of the equivariant Burnside group and define analogous groups, where the symbols involve subgroups, quotient groups, and sequences of characters. 

For example, there would be an analogous group 
$$
\cB_n(T),
$$
generated by unordered sequences of characters 
$$
\beta = (b_1,\ldots, b_n), \quad b_j \in M = \mathfrak X^*(T),
$$
and subject to the blow-up relation as in Section~\ref{ss.symbols}. 

The first issue is that 
$\cB_n(T)$ has infinitely many generators and relations, so that it is not obvious that one can effectively check whether or not two such symbols are equal in $\cB_n(T)$. 

However, there are geometric reasons for not pursuing the formalism for $\Burn_n(T)$. Indeed, $T$-torsors are Zariski locally trivial, and thus a $T$-action on $X$ is equivariantly birational to a $T$-action on $T\times Y$, with $Y=X/T$, standard action of $T$ on itself, 
and trivial action of $T$ on $Y$. Moreover, actions of $T$ on $X$ and $X'$ are equivariantly birational if and only if the quotients $Y$ and $Y'$ are (non-equivariantly) birational. 

\subsection*{Extensions of finite groups by tori}
Another direction would be to consider algebraic groups fitting into an exact sequence
$$
1\to T\to N\to G\to 1, 
$$
where $G$ is a finite group and $T=\bG_m^d$. 

\begin{exam}
\label{exam:cubic}
Consider the singular cubic threefold $X\subset \bP^4$ given by
$$
x_1x_2x_3+x_3^3+x_3x_4x_5+x_4^3+x_5^3=0, 
$$
with singularities at $[1:0:0:0:0]$ and $[0:1:0:0:0]$, and the action of 
$$
1\to \bG_m\to N\to C_2\to 1, 
$$
given by rescaling 
$$
\eta_a\colon x_1\mapsto a x_1, \quad x_2\mapsto a^{-1}x_2
$$ 
and $\tau\colon x_1 \leftrightarrow x_2$, switching the coordinates.  
The fixed locus  is the smooth curve given by $x_1=x_2=0$, of genus 1.  
We can consider $\langle \eta_a, \tau\rangle$, $a\in k^\times$, defining a family of actions of (infinite, for general $a$) dihedral groups on $X$. When $a$ is a primitive even root of unity, the action is not linearizable, by \cite[Prop.\ 5.17]{CMTZ}.
\end{exam}

\subsection*{Reductive groups}
\label{ss.reductive}
One might inquire, whether the theory developed above could be extended to linear algebraic groups, more general than tori.
Reichstein and Youssin have established the existence of a standard form, after suitable equivariant blow-up \cite{RYessential}.
However, several issues arise:
\begin{itemize}
\item The stabilizer groups are extensions by unipotent groups of diagonalizable groups, so their representation theory is not determined by characters.
\item The stabilizer groups can \emph{vary} in a stratum.
\end{itemize}

\begin{exam}
\label{exa.GL2}
Consider the action of $\GL_2$ by right multiplication on the affine space of $2\times 2$-matrices.
The stabilizer groups of
rank $1$ matrices consist of a whole conjugacy class of $\G_a$-extensions of $\G_m$.
As this stratum has codimension $1$,
there is no possibility to improve the situation by blowing up.
\end{exam}

\bibliographystyle{plain}
\bibliography{burnsurv}

\begin{thebibliography}{100}

\bibitem{abban}
H.~Abban, I.~Cheltsov, T.~Kishimoto, and F.~Mangolte.
\newblock Smooth {F}ano 3-folds satisfying {C}ondition {(A)}, 2025.
\newblock {\tt arXiv:2505.13684}.

\bibitem{abramovichtemkin}
D.~Abramovich and M.~Temkin.
\newblock Functorial factorization of birational maps for qe schemes in
  characteristic 0.
\newblock {\em Algebra Number Theory}, 13(2):379--424, 2019.

\bibitem{ant}
B.~Antieau and L.~Meier.
\newblock The {B}rauer group of the moduli stack of elliptic curves.
\newblock {\em Algebra Number Theory}, 14(9):2295--2333, 2020.

\bibitem{ABP}
A.~Auel, Chr. B\"ohning, and A.~Pirutka.
\newblock Stable rationality of quadric and cubic surface bundle fourfolds.
\newblock {\em Eur. J. Math.}, 4(3):732--760, 2018.

\bibitem{ayoub}
J.~Ayoub.
\newblock A guide to (\'etale) motivic sheaves.
\newblock In {\em Proceedings of the {I}nternational {C}ongress of
  {M}athematicians---{S}eoul 2014. {V}ol. {II}}, pages 1101--1124. Kyung Moon
  Sa, Seoul, 2014.

\bibitem{ayoubrealisation}
J.~Ayoub.
\newblock La r\'ealisation \'etale et les op\'erations de {G}rothendieck.
\newblock {\em Ann. Sci. \'Ec. Norm. Sup\'er. (4)}, 47(1):1--145, 2014.

\bibitem{batyrev}
V.~V. Batyrev.
\newblock Birational {C}alabi-{Y}au {$n$}-folds have equal {B}etti numbers.
\newblock In {\em New trends in algebraic geometry ({W}arwick, 1996)}, volume
  264 of {\em London Math. Soc. Lecture Note Ser.}, pages 1--11. Cambridge
  Univ. Press, Cambridge, 1999.

\bibitem{Bcan}
V.~V. Batyrev.
\newblock Canonical abelianization of finite group actions, 2000.
\newblock {\tt arXiv:math/0009043}.

\bibitem{baylebeauville}
L.~Bayle and A.~Beauville.
\newblock Birational involutions of {${\bf P}^2$}.
\newblock {\em Asian J. Math.}, 4(1):11--17, 2000.
\newblock Kodaira's issue.

\bibitem{Beau}
A.~Beauville.
\newblock Vari\'{e}t\'{e}s de {P}rym et jacobiennes interm\'{e}diaires.
\newblock {\em Ann. Sci. \'{E}cole Norm. Sup. (4)}, 10(3):309--391, 1977.

\bibitem{BCSS}
A.~Beauville, J.-L. Colliot-Th\'{e}l\`ene, J.-J. Sansuc, and
  P.~Swinnerton-Dyer.
\newblock Vari\'{e}t\'{e}s stablement rationnelles non rationnelles.
\newblock {\em Ann. of Math. (2)}, 121(2):283--318, 1985.

\bibitem{beke}
T.~Beke.
\newblock Sheafifiable homotopy model categories.
\newblock {\em Math. Proc. Cambridge Philos. Soc.}, 129(3):447--475, 2000.

\bibitem{BO}
O.~Benoist and J.~C. Ottem.
\newblock Two coniveau filtrations.
\newblock {\em Duke Math. J.}, 170(12):2719--2753, 2021.

\bibitem{BW1}
O.~Benoist and O.~Wittenberg.
\newblock The {C}lemens-{G}riffiths method over non-closed fields.
\newblock {\em Algebr. Geom.}, 7(6):696--721, 2020.

\bibitem{BW2}
O.~Benoist and O.~Wittenberg.
\newblock Intermediate {J}acobians and rationality over arbitrary fields.
\newblock {\em Ann. Sci. \'Ec. Norm. Sup\'er. (4)}, 56(4):1029--1084, 2023.

\bibitem{bergh}
D.~Bergh.
\newblock Functorial destackification of tame stacks with abelian stabilisers.
\newblock {\em Compos. Math.}, 153(6):1257--1315, 2017.

\bibitem{berghrydh}
D.~Bergh and D.~Rydh.
\newblock Functorial destackification and weak factorization of orbifolds,
  2019.
\newblock {\tt arXiv:1905.00872}.

\bibitem{bierstonemilman}
E.~Bierstone and P.~D. Milman.
\newblock Canonical desingularization in characteristic zero by blowing up the
  maximum strata of a local invariant.
\newblock {\em Invent. Math.}, 128(2):207--302, 1997.

\bibitem{blanc-thesis}
J.~Blanc.
\newblock Finite abelian subgroups of the {C}remona group of the plane, 2006.
\newblock Ph.D. Thesis, Universit\'e de Gen\`eve, {\tt arXiv:math/0610368}.

\bibitem{blancsubgroups}
J.~Blanc.
\newblock Elements and cyclic subgroups of finite order of the {C}remona group.
\newblock {\em Comment. Math. Helv.}, 86(2):469--497, 2011.

\bibitem{blancfinite}
J.~Blanc, I.~Cheltsov, A.~Duncan, and Yu. Prokhorov.
\newblock Finite quasisimple groups acting on rationally connected threefolds.
\newblock {\em Math. Proc. Cambridge Philos. Soc.}, 174(3):531--568, 2023.

\bibitem{blanc-quo}
J.~Blanc, S.~Lamy, and S.~Zimmermann.
\newblock Quotients of higher-dimensional {C}remona groups.
\newblock {\em Acta Math.}, 226(2):211--318, 2021.

\bibitem{BogPro}
F.~Bogomolov and Yu. Prokhorov.
\newblock On stable conjugacy of finite subgroups of the plane {C}remona group,
  {I}.
\newblock {\em Cent. Eur. J. Math.}, 11(12):2099--2105, 2013.

\bibitem{B-Brauer}
F.~A. Bogomolov.
\newblock The {B}rauer group of quotient spaces of linear representations.
\newblock {\em Izv. Akad. Nauk SSSR Ser. Mat.}, 51(3):485--516, 688, 1987.

\bibitem{bogomolov}
F.~A. Bogomolov.
\newblock Stable cohomology of groups and algebraic varieties.
\newblock {\em Mat. Sb.}, 183(5):3--28, 1992.

\bibitem{BBT}
Chr. B\"ohning, H.-Chr. Graf~v. Bothmer, and Yu. Tschinkel.
\newblock Equivariant birational geometry of cubic fourfolds and derived
  categories.
\newblock {\em Adv. Math.}, 469:Paper No. 110249, 32, 2025.
\newblock With an appendix by Brendan Hassett.

\bibitem{borisov}
L.~A. Borisov.
\newblock The class of the affine line is a zero divisor in the {G}rothendieck
  ring.
\newblock {\em J. Algebraic Geom.}, 27(2):203--209, 2018.

\bibitem{BG}
L.~A. Borisov and P.~E. Gunnells.
\newblock Wonderful blowups associated to group actions.
\newblock {\em Selecta Math. (N.S.)}, 8(3):373--379, 2002.

\bibitem{BLR}
S.~Bosch, W.~L\"utkebohmert, and M.~Raynaud.
\newblock {\em N\'eron models}, volume~21 of {\em Ergebnisse der Mathematik und
  ihrer Grenzgebiete (3) [Results in Mathematics and Related Areas (3)]}.
\newblock Springer-Verlag, Berlin, 1990.

\bibitem{Brandhorst}
S.~Brandhorst and T.~Hofmann.
\newblock Finite subgroups of automorphisms of {K}3 surfaces.
\newblock {\em Forum Math. Sigma}, 11:Paper No. e54, 57, 2023.

\bibitem{brionmodels}
M.~Brion.
\newblock On models of algebraic group actions.
\newblock {\em Proc. Indian Acad. Sci. Math. Sci.}, 132(2):Paper No. 61, 17,
  2022.

\bibitem{CKK}
L.~Cavenaghi, L.~Katzarkov, and M.~Kontsevich.
\newblock Atoms meet symbols, 2025.
\newblock {\tt arXiv:2509.15831}.

\bibitem{CKT}
A.~Chambert-Loir, M.~Kontsevich, and Yu. Tschinkel.
\newblock Burnside rings and volume forms with logarithmic poles, 2023.
\newblock {\tt arXiv:2301.02899}, to appear in {\em PAMQ}.

\bibitem{CLNS}
A.~Chambert-Loir, J.~Nicaise, and J.~Sebag.
\newblock {\em Motivic integration}, volume 325 of {\em Progress in
  Mathematics}.
\newblock Birkh\"{a}user/Springer, New York, 2018.

\bibitem{CMTZ}
I.~Cheltsov, L.~Marquand, Yu. Tschinkel, and Zh. Zhang.
\newblock Equivariant geometry of singular cubic threefolds, {II}.
\newblock {\em J. Lond. Math. Soc. (2)}, 112(1):Paper No. e70224, 46, 2025.

\bibitem{CSZ}
I.~Cheltsov, A.~Sarikyan, and Z.~Zhuang.
\newblock Birational rigidity and alpha invariants of {F}ano varieties.
\newblock In {\em Higher dimensional algebraic geometry---a volume in honor of
  {V}. {V}. {S}hokurov}, volume 489 of {\em London Math. Soc. Lecture Note
  Ser.}, pages 286--318. Cambridge Univ. Press, Cambridge, 2025.

\bibitem{CS}
I.~Cheltsov and C.~Shramov.
\newblock {\em Cremona groups and the icosahedron}.
\newblock Monographs and Research Notes in Mathematics. CRC Press, Boca Raton,
  FL, 2016.

\bibitem{CTZ-p2}
I.~Cheltsov, Yu. Tschinkel, and Zh. Zhang.
\newblock Conjugacy classes of linear actions in the plane {C}remona group,
  2025.
\newblock {\tt arXiv:2508.09929}.

\bibitem{CTZ-forum}
I.~Cheltsov, Yu. Tschinkel, and Zh. Zhang.
\newblock Equivariant geometry of singular cubic threefolds.
\newblock {\em Forum Math. Sigma}, 13:Paper No. e9, 52, 2025.

\bibitem{CTZ-burk}
I.~Cheltsov, Yu. Tschinkel, and Zh. Zhang.
\newblock Equivariant geometry of the {S}egre cubic and the {B}urkhardt
  quartic.
\newblock {\em Selecta Math. (N.S.)}, 31(1):Paper No. 7, 36, 2025.

\bibitem{CTZ-3}
I.~Cheltsov, Yu. Tschinkel, and Zh. Zhang.
\newblock Equivariant unirationality of {F}ano threefolds, 2025.
\newblock {\tt 2502.19598}.

\bibitem{CTZ-real}
I.~Cheltsov, Yu. Tschinkel, and Zh. Zhang.
\newblock Rationality of singular cubic threefolds over {$\mathbb R$}.
\newblock {\em Adv. Math.}, 487:Paper No. 110756, 2026.

\bibitem{CGR}
V.~Chernousov, P.~Gille, and Z.~Reichstein.
\newblock Resolving {$G$}-torsors by abelian base extensions.
\newblock {\em J. Algebra}, 296(2):561--581, 2006.

\bibitem{choudhurygallauer}
U.~Choudhury and M.~Gallauer Alves~de Souza.
\newblock Homotopy theory of dg sheaves.
\newblock {\em Comm. Algebra}, 47(8):3202--3228, 2019.

\bibitem{CG}
C.~H. Clemens and P.~A. Griffiths.
\newblock The intermediate {J}acobian of the cubic threefold.
\newblock {\em Ann. of Math. (2)}, 95:281--356, 1972.

\bibitem{CT-schrei}
J.-L. Colliot-Th\'el\`ene.
\newblock Introduction to work of {H}assett-{P}irutka-{T}schinkel and
  {S}chreieder.
\newblock In {\em Birational geometry of hypersurfaces}, volume~26 of {\em
  Lect. Notes Unione Mat. Ital.}, pages 111--125. Springer, Cham, 2019.

\bibitem{CTP}
J.-L. Colliot-Th\'{e}l\`ene and A.~Pirutka.
\newblock Hypersurfaces quartiques de dimension 3: non-rationalit\'{e} stable.
\newblock {\em Ann. Sci. \'{E}c. Norm. Sup\'{e}r. (4)}, 49(2):371--397, 2016.

\bibitem{CTSansucDuke}
J.-L. Colliot-Th\'{e}l\`ene and J.-J. Sansuc.
\newblock La descente sur les vari\'{e}t\'{e}s rationnelles. {II}.
\newblock {\em Duke Math. J.}, 54(2):375--492, 1987.

\bibitem{conradgabberprasad}
B.~Conrad, O.~Gabber, and G.~Prasad.
\newblock {\em Pseudo-reductive groups}, volume~26 of {\em New Mathematical
  Monographs}.
\newblock Cambridge University Press, Cambridge, second edition, 2015.

\bibitem{DM}
P.~Deligne and D.~Mumford.
\newblock The irreducibility of the space of curves of given genus.
\newblock {\em Inst. Hautes \'Etudes Sci. Publ. Math.}, (36):75--109, 1969.

\bibitem{deninger}
Ch. Deninger.
\newblock A proper base change theorem for nontorsion sheaves in \'etale
  cohomology.
\newblock {\em J. Pure Appl. Algebra}, 50(3):231--235, 1988.

\bibitem{DI}
I.~V. Dolgachev and V.~A. Iskovskikh.
\newblock Finite subgroups of the plane {C}remona group.
\newblock In {\em Algebra, arithmetic, and geometry: in honor of {Y}u. {I}.
  {M}anin. {V}ol. {I}}, volume 269 of {\em Progr. Math.}, pages 443--548.
  Birkh\"{a}user Boston, Boston, MA, 2009.

\bibitem{EG}
D.~Edidin and W.~Graham.
\newblock Equivariant intersection theory.
\newblock {\em Invent. Math.}, 131(3):595--634, 1998.

\bibitem{GKR}
R.~Garner, M.~K{\c e}dziorek, and E.~Riehl.
\newblock Lifting accessible model structures.
\newblock {\em J. Topol.}, 13(1):59--76, 2020.

\bibitem{giraud}
J.~Giraud.
\newblock {\em Cohomologie non ab\'{e}lienne}.
\newblock Die Grundlehren der mathematischen Wissenschaften, Band 179.
  Springer-Verlag, Berlin-New York, 1971.

\bibitem{HM}
Ch.~D. Hacon, J.~McKernan, and Ch. Xu.
\newblock On the birational automorphisms of varieties of general type.
\newblock {\em Ann. of Math. (2)}, 177(3):1077--1111, 2013.

\bibitem{hassetthyeon}
B.~Hassett and D.~Hyeon.
\newblock Log canonical models for the moduli space of curves: the first
  divisorial contraction.
\newblock {\em Trans. Amer. Math. Soc.}, 361(8):4471--4489, 2009.

\bibitem{HKTconic}
B.~Hassett, A.~Kresch, and Yu. Tschinkel.
\newblock Stable rationality and conic bundles.
\newblock {\em Math. Ann.}, 365(3-4):1201--1217, 2016.

\bibitem{HKTsmall}
B.~Hassett, A.~Kresch, and Yu. Tschinkel.
\newblock Symbols and equivariant birational geometry in small dimensions.
\newblock In {\em Rationality of varieties}, volume 342 of {\em Progr. Math.},
  pages 201--236. Birkh\"{a}user, Cham, 2021.

\bibitem{HPT-3quad}
B.~Hassett, A.~Pirutka, and Yu. Tschinkel.
\newblock Intersections of three quadrics in {$\mathbb P^7$}.
\newblock In {\em Surveys in differential geometry 2017. {C}elebrating the 50th
  anniversary of the {J}ournal of {D}ifferential {G}eometry}, volume~22 of {\em
  Surv. Differ. Geom.}, pages 259--274. Int. Press, Somerville, MA, 2018.

\bibitem{HPT-quadric}
B.~Hassett, A.~Pirutka, and Yu. Tschinkel.
\newblock Stable rationality of quadric surface bundles over surfaces.
\newblock {\em Acta Math.}, 220(2):341--365, 2018.

\bibitem{HPT-quartic}
B.~Hassett, A.~Pirutka, and Yu. Tschinkel.
\newblock A very general quartic double fourfold is not stably rational.
\newblock {\em Algebr. Geom.}, 6(1):64--75, 2019.

\bibitem{HT-Fano}
B.~Hassett and Yu. Tschinkel.
\newblock On stable rationality of {F}ano threefolds and del {P}ezzo
  fibrations.
\newblock {\em J. Reine Angew. Math.}, 751:275--287, 2019.

\bibitem{HT-cycle}
B.~Hassett and Yu. Tschinkel.
\newblock Cycle class maps and birational invariants.
\newblock {\em Comm. Pure Appl. Math.}, 74(12):2675--2698, 2021.

\bibitem{HT-rat}
B.~Hassett and Yu. Tschinkel.
\newblock Rationality of complete intersections of two quadrics over nonclosed
  fields.
\newblock {\em Enseign. Math.}, 67(1-2):1--44, 2021.
\newblock With an appendix by Jean-Louis Colliot-Th\'{e}l\`ene.

\bibitem{HTtors}
B.~Hassett and Yu. Tschinkel.
\newblock Torsors and stable equivariant birational geometry.
\newblock {\em Nagoya Math. J.}, 250:275--297, 2023.

\bibitem{HT-quad}
B.~Hassett and Yu. Tschinkel.
\newblock Equivariant geometry of low-dimensional quadrics.
\newblock {\em Taiwanese J. Math.}, 29(6):1381--1402, 2025.

\bibitem{hoshiyamasaki}
A.~Hoshi and A.~Yamasaki.
\newblock Rationality problem for algebraic tori.
\newblock {\em Mem. Amer. Math. Soc.}, 248(1176):v+215, 2017.

\bibitem{hoveyreptheory}
M.~Hovey.
\newblock Cotorsion pairs, model category structures, and representation
  theory.
\newblock {\em Math. Z.}, 241(3):553--592, 2002.

\bibitem{huberkahn}
A.~Huber and B.~Kahn.
\newblock The slice filtration and mixed {T}ate motives.
\newblock {\em Compos. Math.}, 142(4):907--936, 2006.

\bibitem{Huy}
D.~Huybrechts.
\newblock {\em Lectures on {K}3 surfaces}, volume 158 of {\em Cambridge Studies
  in Advanced Mathematics}.
\newblock Cambridge University Press, Cambridge, 2016.

\bibitem{iitaka}
S.~Iitaka.
\newblock {\em Algebraic geometry}, volume~24 of {\em North-Holland
  Mathematical Library}.
\newblock Springer-Verlag, New York-Berlin, 1982.
\newblock An introduction to birational geometry of algebraic varieties,
  Graduate Texts in Mathematics, 76.

\bibitem{isk-s3}
V.~A. Iskovskikh.
\newblock Two non-conjugate embeddings of {$S_3\times Z_2$} into the {C}remona
  group. {II}.
\newblock In {\em Algebraic geometry in {E}ast {A}sia---{H}anoi 2005},
  volume~50 of {\em Adv. Stud. Pure Math.}, pages 251--267. Math. Soc. Japan,
  Tokyo, 2008.

\bibitem{MI-lu}
V.~A. Iskovskikh and Yu.~I. Manin.
\newblock Three-dimensional quartics and counterexamples to the {L}\"uroth
  problem.
\newblock {\em Mat. Sb. (N.S.)}, 86(128):140--166, 1971.

\bibitem{KN-I}
B.~Kahn and Nguyen T.~K. Ngan.
\newblock Sur l'espace classifiant d'un groupe alg\'{e}brique lin\'{e}aire,
  {I}.
\newblock {\em J. Math. Pures Appl. (9)}, 102(5):972--1013, 2014.

\bibitem{kameko}
M.~Kameko.
\newblock Coniveau filtrations with {$\bZ/2$}-coefficients, 2025.
\newblock {\tt arXiv:2504.19388}.

\bibitem{P-cyclic}
Zh.-L. Ko{$\mathrm{l}^\prime$}\"e-Tel\`en and E.~V. Piryutko.
\newblock Cyclic covers that are not stably rational.
\newblock {\em Izv. Ross. Akad. Nauk Ser. Mat.}, 80(4):35--48, 2016.

\bibitem{KPT}
M.~Kontsevich, V.~Pestun, and Yu. Tschinkel.
\newblock Equivariant birational geometry and modular symbols.
\newblock {\em J. Eur. Math. Soc. (JEMS)}, 25(1):153--202, 2023.

\bibitem{KT}
M.~Kontsevich and Yu. Tschinkel.
\newblock Specialization of birational types.
\newblock {\em Invent. Math.}, 217(2):415--432, 2019.

\bibitem{KTT}
A.~Kresch, S.~Tanimoto, and Yu. Tschinkel.
\newblock Intermediate {J}acobians and {B}urnside invariants, 2025.
\newblock {\tt arXiv:2511.07101}, to appear in J. Math. Soc. Japan.

\bibitem{KT-SRBS}
A.~Kresch and Yu. Tschinkel.
\newblock Stable rationality of {B}rauer-{S}everi surface bundles.
\newblock {\em Manuscripta Math.}, 161(1-2):1--14, 2020.

\bibitem{Bbar}
A.~Kresch and Yu. Tschinkel.
\newblock Birational types of algebraic orbifolds.
\newblock {\em Mat. Sb.}, 212(3):54--67, 2021.

\bibitem{KT-dp}
A.~Kresch and Yu. Tschinkel.
\newblock Cohomology of finite subgroups of the plane {C}remona group, 2022.
\newblock {\tt arXiv:2203.01876}, to appear in Algebraic Geom. and Physics.

\bibitem{BnG}
A.~Kresch and Yu. Tschinkel.
\newblock Equivariant birational types and {B}urnside volume.
\newblock {\em Ann. Sc. Norm. Super. Pisa Cl. Sci. (5)}, 23(2):1013--1052,
  2022.

\bibitem{KT-vector}
A.~Kresch and Yu. Tschinkel.
\newblock Equivariant {B}urnside groups and representation theory.
\newblock {\em Selecta Math. (N.S.)}, 28(4):Paper No. 81, 39, 2022.

\bibitem{KT-stacks}
A.~Kresch and Yu. Tschinkel.
\newblock Birational geometry of {D}eligne-{M}umford stacks, 2023.
\newblock {\tt arXiv:2312.14061}, to appear in Ann. Sc. Norm. Super. Pisa Cl.
  Sci.

\bibitem{KT-toric}
A.~Kresch and Yu. Tschinkel.
\newblock Equivariant {B}urnside groups and toric varieties.
\newblock {\em Rend. Circ. Mat. Palermo (2)}, 72(5):3013--3039, 2023.

\bibitem{KT-struct}
A.~Kresch and Yu. Tschinkel.
\newblock Equivariant {B}urnside groups: structure and operations.
\newblock {\em Taiwanese J. Math.}, 29(6):1507--1529, 2025.

\bibitem{KT-uni}
A.~Kresch and Yu. Tschinkel.
\newblock Equivariant unirationality of toric varieties, 2025.
\newblock {\tt arXiv:2506.07152}.

\bibitem{KT-quotient}
A.~Kresch and Yu. Tschinkel.
\newblock Unramified {B}rauer group of quotient spaces by finite groups.
\newblock {\em J. Algebra}, 664:75--100, 2025.

\bibitem{KuP1}
A.~Kuznetsov and Yu. Prokhorov.
\newblock Rationality of {F}ano threefolds over non-closed fields.
\newblock {\em Amer. J. Math.}, 145(2):335--411, 2023.

\bibitem{KuP2}
A.~Kuznetsov and Yu. Prokhorov.
\newblock Rationality over nonclosed fields of {F}ano threefolds with higher
  geometric {P}icard rank.
\newblock {\em J. Inst. Math. Jussieu}, 23(1):207--247, 2024.

\bibitem{larsenlunts}
M.~Larsen and V.~A. Lunts.
\newblock Motivic measures and stable birational geometry.
\newblock {\em Mosc. Math. J.}, 3(1):85--95, 259, 2003.

\bibitem{manin}
Yu.~I. Manin.
\newblock Rational surfaces over perfect fields.
\newblock {\em Inst. Hautes \'{E}tudes Sci. Publ. Math.}, (30):55--113, 1966.

\bibitem{manin2}
Yu.~I. Manin.
\newblock Rational surfaces over perfect fields. {II}.
\newblock {\em Mat. Sb. (N.S.)}, 72 (114):161--192, 1967.

\bibitem{weibel}
C.~Mazza, V.~Voevodsky, and C.~Weibel.
\newblock {\em Lecture notes on motivic cohomology}, volume~2 of {\em Clay
  Mathematics Monographs}.
\newblock American Mathematical Society, Providence, RI; Clay Mathematics
  Institute, Cambridge, MA, 2006.

\bibitem{milne}
J.~S. Milne.
\newblock {\em \'{E}tale cohomology}.
\newblock Princeton Mathematical Series, No. 33. Princeton University Press,
  Princeton, N.J., 1980.

\bibitem{mumford-git}
D.~Mumford, J.~Fogarty, and F.~Kirwan.
\newblock {\em Geometric invariant theory}, volume~34 of {\em Ergeb. Math.
  Grenzgeb.}
\newblock Berlin: Springer-Verlag, 3rd enl. edition, 1994.

\bibitem{NO-refinement}
J.~Nicaise and J.~C. Ottem.
\newblock A refinement of the motivic volume, and specialization of birational
  types.
\newblock In {\em Rationality of varieties}, volume 342 of {\em Progr. Math.},
  pages 291--322. Birkh\"auser/Springer, Cham, 2021.

\bibitem{NO-tropical}
J.~Nicaise and J.~C. Ottem.
\newblock Tropical degenerations and stable rationality.
\newblock {\em Duke Math. J.}, 171(15):3023--3075, 2022.

\bibitem{nicaiseshinder}
J.~Nicaise and E.~Shinder.
\newblock The motivic nearby fiber and degeneration of stable rationality.
\newblock {\em Invent. Math.}, 217(2):377--413, 2019.

\bibitem{peyre}
E.~Peyre.
\newblock Progr\`es en irrationalit\'e{} [d'apr\`es {C}. {V}oisin, {J}.-{L}.
  {C}olliot-{T}h\'el\`ene, {B}. {H}assett, {A}. {K}resch, {A}. {P}irutka, {B}.
  {T}otaro, {Y}. {T}schinkel et al.].
\newblock Number 407, pages Exp. No. 1123, 91--116. 2019.
\newblock S\'eminaire Bourbaki. Vol. 2016/2017. Expos\'es 1120--1135.

\bibitem{pinardin}
A.~Pinardin, A.~Sarikyan, and E.~Yasinsky.
\newblock Linearization problem for finite subgroups of the plane {C}remona
  group, 2024.
\newblock {\tt arXiv:2412.12022}.

\bibitem{PZ}
A.~Pinardin and Zh. Zhang.
\newblock {$\mathfrak A_5$}-equivariant geometry of quadric threefolds, 2025.
\newblock {\tt arXiv:2508.11496}.

\bibitem{ProSimple}
Yu. Prokhorov.
\newblock Simple finite subgroups of the {C}remona group of rank 3.
\newblock {\em J. Algebraic Geom.}, 21(3):563--600, 2012.

\bibitem{Pro-ECM}
Yu~Prokhorov.
\newblock Finite groups of birational transformations.
\newblock In {\em European {C}ongress of {M}athematics ({P}ortoro{\v z},
  2021)}, pages 413--437. EMS Press, Berlin, 2023.

\bibitem{pro-jordan}
Yu. Prokhorov and C.~Shramov.
\newblock Jordan property for {C}remona groups.
\newblock {\em Amer. J. Math.}, 138(2):403--418, 2016.

\bibitem{pukh}
A.~Pukhlikov.
\newblock {\em Birationally rigid varieties}, volume 190 of {\em Mathematical
  Surveys and Monographs}.
\newblock American Mathematical Society, Providence, RI, 2013.

\bibitem{RYessential}
Z.~Reichstein and B.~Youssin.
\newblock Essential dimensions of algebraic groups and a resolution theorem for
  {$G$}-varieties.
\newblock {\em Canad. J. Math.}, 52(5):1018--1056, 2000.
\newblock With an appendix by J\'{a}nos Koll\'{a}r and Endre Szab\'{o}.

\bibitem{RYinvariant}
Z.~Reichstein and B.~Youssin.
\newblock A birational invariant for algebraic group actions.
\newblock {\em Pacific J. Math.}, 204(1):223--246, 2002.

\bibitem{Schr-2}
S.~Schreieder.
\newblock On the rationality problem for quadric bundles.
\newblock {\em Duke Math. J.}, 168(2):187--223, 2019.

\bibitem{Schr-small-slopes}
S.~Schreieder.
\newblock Stably irrational hypersurfaces of small slopes.
\newblock {\em J. Amer. Math. Soc.}, 32(4):1171--1199, 2019.

\bibitem{sharma}
Sh. Sharma.
\newblock Actions on the {P}icard group of smooth {F}ano threefolds, 2025.
\newblock {\tt arXiv:2511.12447}.

\bibitem{spaltenstein}
N.~Spaltenstein.
\newblock Resolutions of unbounded complexes.
\newblock {\em Compositio Math.}, 65(2):121--154, 1988.

\bibitem{totaro-hyper}
B.~Totaro.
\newblock Hypersurfaces that are not stably rational.
\newblock {\em J. Amer. Math. Soc.}, 29(3):883--891, 2016.

\bibitem{tschinkelyangzhang}
Yu. Tschinkel, K.~Yang, and Zh. Zhang.
\newblock Combinatorial {B}urnside groups.
\newblock {\em Res. Number Theory}, 8(2):Paper No. 33, 17, 2022.

\bibitem{TYZ-survey}
Yu. Tschinkel, K.~Yang, and Zh. Zhang.
\newblock Equivariant birational geometry of linear actions.
\newblock {\em EMS Surv. Math. Sci.}, 11(2):235--276, 2024.

\bibitem{TZ-uni}
Yu. Tschinkel and Zh. Zhang.
\newblock Cohomological obstructions to equivariant unirationality, 2025.
\newblock {\tt arXiv:2504.10204}.

\bibitem{TZ-toric}
Yu. Tschinkel and Zh. Zhang.
\newblock Equivariant unirationality of tori in small dimensions, 2025.
\newblock {\tt arXiv:2509.17008}.

\bibitem{TZ-14}
Yu. Tschinkel and Zh. Zhang.
\newblock Stable equivariant birationalities of cubic and degree 14 {F}ano
  threefolds, 2025.
\newblock {\tt arXiv:2409.08392}, to appear in Algebraic Geom. and Physics.

\bibitem{voevodcohomological}
V.~Voevodsky.
\newblock Cohomological theory of presheaves with transfers.
\newblock In {\em Cycles, transfers, and motivic homology theories}, volume 143
  of {\em Ann. of Math. Stud.}, pages 87--137. Princeton Univ. Press,
  Princeton, NJ, 2000.

\bibitem{voevodtriangulated}
V.~Voevodsky.
\newblock Triangulated categories of motives over a field.
\newblock In {\em Cycles, transfers, and motivic homology theories}, volume 143
  of {\em Ann. of Math. Stud.}, pages 188--238. Princeton Univ. Press,
  Princeton, NJ, 2000.

\bibitem{voisin}
C.~Voisin.
\newblock Unirational threefolds with no universal codimension {$2$} cycle.
\newblock {\em Invent. Math.}, 201(1):207--237, 2015.

\end{thebibliography}

\end{document}